\documentclass[a4paper,UKenglish,cleveref, numberwithinsect, thm-restate,authorcolumns]{lipics-v2019_modified}

\usepackage{amssymb,amsthm}
\usepackage{amsfonts}
\usepackage{amsmath}
\usepackage{mathtools}
\usepackage{cancel}
\usepackage{stmaryrd}
\SetSymbolFont{stmry}{bold}{U}{stmry}{m}{n}
\usepackage{tikz}
\usetikzlibrary{calc}
\usetikzlibrary{automata, positioning}
\usepackage{float}
\usepackage{latexsym}
\usepackage{etoolbox}
\usepackage{cutwin}
\usepackage{mathrsfs}
\usepackage{anyfontsize}
\usepackage{microtype}

\DeclareMathOperator{\A}{A}
\DeclareMathOperator{\D}{D}

\DeclareMathOperator{\R}{R}
\DeclareMathOperator{\dd}{\mathbf d}
\DeclareMathOperator{\rr}{\mathbf r}

\DeclareMathOperator{\id}{1^\prime}

\newcommand \compo {\mathbin{;}}
\newcommand \ccompo {\mathbin{\cdot}}
\newcommand \pref {\mathbin{\sqcup}}

\newcommand \bjoin {\vee}

\newcommand \limplies {\mathbin{\rightarrow}}

\newcommand \aand \wedge

\newcommand{\from}{\colon}

\newcommand{\class}[1]{\mathbf{#1}}

\newcommand{\algebra}[1]{\mathfrak{#1}}
\newcommand{\topo}[1]{\mathcal{#1}}

\newcommand{\defn}[1]{\textbf{#1}}

\newcommand\pf{\operatorname{PF}}
\newcommand\cpt{\operatorname{SecCl}}

\makeatletter
\newcommand\mynobreakpar{\par\nobreak\@afterheading}
\makeatother

\newtheorem{problem}[theorem]{Problem}
\newtheorem{lemma3}[theorem]{Lemma}
\newtheorem{lemma4}[theorem]{Lemma}
\newtheorem{lemma5}[theorem]{Lemma}
\newtheorem{lemma6}[theorem]{Lemma}
\newtheorem{definition4}[theorem]{Definition}

\title{A categorical duality for algebras of partial functions} 

\titlerunning{A categorical duality for algebras of partial functions} 

\author{Brett McLean}{Laboratoire J. A. Dieudonn\'e UMR CNRS 7351, Universit\'e Nice Sophia Antipolis, 06108~Nice~Cedex~02 \and \url{https://math.unice.fr/~bmclean/} }{brett.mclean@unice.fr}{https://orcid.org/0000-0003-2368-8357}{}

\authorrunning{B. McLean} 

\Copyright{Brett McLean} 

\ccsdesc[500]{Primary: 06F05}
\ccsdesc[300]{Secondary: 20M30, 06E15, 20M35} 

\keywords{partial function, duality, Stone space, finite state transducer} 

\category{} 

\supplement{}

\acknowledgements{The author would like to thank Mai Gehrke for suggesting the project of a partial function duality building on the ideas in \cite{Lawson_2010} and for several helpful discussions.}

\funding{This project has received funding from the European Research Council (ERC) under the European Union's
Horizon 2020 research and innovation program (grant agreement No. 670624).}

\nolinenumbers 

\hideLIPIcs  

\EventEditors{John Q. Open and Joan R. Access}
\EventNoEds{2}
\EventLongTitle{42nd Conference on Very Important Topics (CVIT 2016)}
\EventShortTitle{CVIT 2016}
\EventAcronym{CVIT}
\EventYear{2016}
\EventDate{December 24--27, 2016}
\EventLocation{Little Whinging, United Kingdom}
\EventLogo{}
\SeriesVolume{42}
\ArticleNo{23}

\begin{document}

\maketitle

\begin{abstract}
We prove a categorical duality between a class of abstract algebras of partial functions and a class of (small) topological categories. The algebras are the isomorphs of collections of partial functions closed under the operations of composition, antidomain, range, and preferential union (or `override'). The topological categories are those whose space of objects is a Stone space, source map is a local homeomorphism, target map is open, and all of whose arrows are epimorphisms. 
\end{abstract}

\section{Introduction}

Variants and extensions of Stone duality are pervasive in logic and computer science: for example in modal \cite{goldblatt}, intuitionistic \cite{esakia}, substructural \cite{10.2307/27588391}, and many-valued \cite{GEHRKE2014290} logic, and in semantics \cite{ABRAMSKY19911}, formal language theory \cite{gehrke2008duality}, and logics for static analysis \cite{DOCHERTY2018101}. In its most basic form---between Boolean algebras and Stone spaces---it provides a duality for the isomorphs of \emph{algebras of unary relations} equipped with the union and complement operations. One extremely prominent `real world' example of such an algebra is the set of regular languages (for a fixed finite alphabet). Indeed, recently it has been shown how the presence of an extended Stone duality is the explanation behind many of the great successes of algebraic language theory \cite{gehrke2008duality, gehrke2009stone}. 

Another genre of algebras of relations that has been studied is \emph{algebras of partial functions} \cite{1018.20057, DBLP:journals/ijac/JacksonS11, hirsch, phdthesis}. Here the algebras are formed of collections of partial functions closed under certain natural operations such as composition or `preferential union'. Again, pre-existing examples can be found in automata/formal language theory. \emph{Transducers} are finite state machines that take words as input and produce words as output. In general they realise word-to-word relations, but certain collections of word-to-word partial functions defined via transducers are considered important---we think particularly of the \emph{rational functions} and the \emph{regular functions} \cite{10.1145/2984450.2984453}.

In this paper, motivated by the utility of duality applied to regular languages, we give a description of a duality (\Cref{main-theorem}) applicable to both the rational functions and the regular functions. Specifically, on one side is the category of isomorphs of algebras of partial functions equipped with four operations: \emph{composition}, \emph{antidomain}, \emph{range}, and \emph{preferential union} (see \Cref{preliminaries}). On the other side is the category of (small) categories each equipped with a topology and satisfying certain extra conditions (see \Cref{category_definitions}). This duality is a partial function analogue of a duality due to Mark Lawson between a certain subclass of inverse semigroups and certain topological groupoids \cite{Lawson_2010}. Recall that the inverse semigroups are the isomorphs of \emph{injective} partial functions. In fact the duality presented here not only broadens the scope from injective to arbitrary partial functions, but also generalises which morphisms of algebras are handled. In our duality, the morphisms of algebras are exactly the homomorphisms, whereas in \cite{Lawson_2010} they are a restricted type of homomorphism only.

Given that partial functions are just a special type of binary relation, we are compelled to acknowledge the numerous dualities applicable to, or even designed specifically for, algebras of binary relations. The algebras featuring in these dualities span a range from the very highly structured Boolean algebras with operators \cite[Chapter 5: Algebras and General Frames]{blackburn_rijke_venema_2001} down to the much more general cases of (bounded) distributive lattices with completely arbitrary additional operations \cite{10.2307/24493402} and of posets with monotone operations \cite{10.2307/27588391}. Two remarks are in order here. Firstly, none of the mentioned dualities are applicable to our algebras, since on the one hand they are not distributive lattices, but on the other hand not all their operations are monotone. Secondly, the use of a general theory would not be optimal in any case, since our algebras \emph{do} possess a great deal of structure, reflecting their concrete origins.

\medskip

\noindent{\bf Structure of the paper}\ \ In \Cref{preliminaries} we formally define our algebras of functions and the category they constitute. In \Cref{category_definitions} we do the same for the small topological categories on the other side of our duality. In \Cref{algebra_to_category}, we describe one half of the duality: the functor from algebras to topological categories. In \Cref{category_to_algebra}, we describe the remaining half of the duality: the functor from topological categories to algebras. In \Cref{are_dual}, we prove that these two functors do indeed form a contravariant equivalence of categories. In \Cref{applications}, we say a little about the duality as it applies to the rational and regular functions. 

\section{Algebras of functions}\label{preliminaries}

Given an algebra $\algebra{A}$, when we write $a \in \algebra{A}$ or say that $a$ is an element of $\algebra{A}$, we mean that $a$ is an element of the domain of $\algebra{A}$. Similarly for the notation $S \subseteq \algebra{A}$ or saying that $S$ is a subset of $\algebra{A}$. We follow the convention that algebras are always nonempty. If $S$ is a subset of the domain of a map $\theta$ then $\theta[S]$ denotes the set $\{\theta(s) \mid s \in S\}$. As is common in algebraic logic, compositions denoted with the symbol $\compo$ are written with the first composee on the left, that is, contrary to the usual mathematical convention. We will, however, also use the conventional $\circ$ notation with the conventional ordering in situations having no connection to the composition of partial functions. If $S$ and $T$ are subsets of $\algebra A$, then we abuse notation by writing $S \compo T$ for $\{s \compo t \mid s \in S\text{ and }t \in T\}$ and abuse further by writing $S \compo a$ and $a \compo S$ for $S \compo \{a\}$ and $\{a\} \compo S$ respectively.

We begin by making precise what is meant by partial functions and algebras of partial functions.

\begin{definition}
Let $X$ be a set. A \defn{partial function} on $X$ is a subset $f$ of $X \times X$ satisfying\begin{equation*}
(x, y) \in f \text{ and } (x, z) \in f \implies y = z.
\end{equation*}
\end{definition}

\begin{definition}\label{algebra}
Let $\sigma \subseteq \{\compo, \A, \R, \pref\}$ be a functional signature, where $\compo$ and $\pref$ are binary and $\A$ and $\R$ are unary. An \defn{algebra of partial functions} of the signature $\sigma$ is a universal algebra $\algebra A = (A, \sigma)$ where the elements of the universe $A$ are all partial functions on some (common) set $X$, the \defn{base}, and the interpretations of the symbols are given as follows.
\begin{itemize}
\item
The binary operation $\compo$ is \defn{composition} of partial functions.

\item
The unary operation $\A$ is the operation of taking the diagonal of the \defn{antidomain} of a partial function:
\[\A(f) \coloneqq \{(x, x) \in X^2 \mid \not\exists y \in X : (x, y) \in f\}\text{.}\]

\item
The unary operation $\R$ is the operation of taking the diagonal of the \defn{range} of a partial function:
\[\R(f) \coloneqq \{(y, y) \in X^2 \mid \exists x \in X : (x, y) \in f\}\text{.}\]
\item
the binary operation $\pref$ is \defn{preferential union}:
\[(f \pref g)(x) =
\begin{cases}
f(x) & \text{if }f(x)\text{ defined}\\
g(x) & \text{if }f(x)\text{ undefined, but }g(x)\text{ defined}\\
\text{undefined} & \text{otherwise}
\end{cases}\]
\end{itemize}
\end{definition}

Note that (despite the symmetry of the symbol $\pref$) preferential union is not generally a commutative operation, though it is associative.

\begin{definition}
An algebra $\algebra A$ of the signature $\sigma$ is \defn{representable} by partial functions if it is isomorphic to an algebra of partial functions of the signature $\sigma$. An isomorphism from $\algebra A$ to an algebra of partial functions is a \defn{representation} of $\algebra A$.
\end{definition}

We begin by looking at representable $\{\compo, \A, \R\}$-algebras, but we will soon see that the algebras we are interested in are equivalent to the representable $\{\compo, \A, \R, \pref\}$-algebras.

\begin{remark}
Note that the constants $0$ (empty function) and $\id$ (identity function), the operation $\D$ (domain), and the relation $\leq$ (subset) are all definable in the signature $\{\compo, \A, \R\}$. That is, the term $0 \coloneqq \A(a) \compo a$, the term $\id \coloneqq \A(0)$, and the term $\D(a) \coloneqq \A(\A(a))$ are all necessarily represented in the intended way by any representation, and the relation 
\begin{equation*}\label{ordering}a \leq b \iff \D(a) \compo b = a\end{equation*}
 corresponds via any representation precisely to the subset relation on the image of the representation.
\end{remark}

Statements involving order will always be with respect to $\leq$.

\begin{remark}
The representable $\{\compo, \A, \R\}$-algebras form a proper quasivariety, axiomatised by a finite number of quasiequations \cite[Theorem~4.1]{hirsch}.
\end{remark}

\begin{definition}
Two elements $a$ and $b$ of an algebra of the signature $\{\compo, \A, \R\}$ are \defn{compatible} if $\D(a) \compo b = \D(b) \compo a$.
\end{definition}

Clearly in any representable $\{\compo, \A, \R\}$-algebra the compatibility relation expresses precisely that for any representation the representing functions agree on their common domain. In such an algebra compatibility is necessary for the existence of a least upper bound of a pair $a$ and $b$. If $a$ and $b$ have an upper bound, $c$ say, they have a least upper bound (join) given by $\A(\A(a) \compo \A(b)) \compo c$. From this term, we see that in concrete algebras any binary joins must be given by binary unions.

\begin{lemma}\label{preserves}
Homomorphisms of representable $\{\compo, \A, \R\}$-algebras preserve binary joins.
\end{lemma}

\begin{proof}
Let $h \from \algebra A \to \algebra B$ be a homomorphism of representable $\{\compo, \A, \R\}$-algebras, and suppose $a, b \in \algebra A$ have an upper bound. Since the algebras are representable, we may assume they are algebras of partial functions. As $\leq$ is defined by an equation, $h$ is order preserving, so $h(a \bjoin b)$ is an upper bound for $\{h(a), h(b)\}$, that is, $h(a), h(b) \subseteq h(a \bjoin b)$. On the other hand
\begin{align*}
h(a \bjoin b) &= h((\D(a)\bjoin \D(b)) \compo (a \bjoin b))\\
&= (\D(h(a)) \cup \D(h(b))) \compo h(a \bjoin b).
\end{align*}
(We know joins correspond to unions on the subalgebra of elements of the form $\D(-)$, since join is expressible there, as $\A(\A(-) \compo \A (-))$.) Hence $h(a \bjoin b) \subseteq h(a) \cup h(b)$, and so $h(a \bjoin b)$ is the smallest possible upper bound for $h(a)$ and $h(b)$, namely $h(a) \cup h(b)$.
\end{proof}

Let $\class A$ be the subclass of the representable $\{\compo, \A, \R\}$-algebras consisting of those validating the first-order condition that every compatible pair has an upper bound.

\begin{corollary}\label{isomorphic}
The category consisting of $\class A$ with $\{\compo, \A, \R\}$-homomorphisms is isomorphic to the category of representable $\{\compo, \A, \R, \pref\}$-algebras with $\{\compo, \A, \R, \pref\}$-homo\-morphisms.
\end{corollary}

\begin{proof}
In any representable $\{\compo, \A, \R\}$-algebra in which compatible pairs have upper bounds, the operation $\pref$ is definable as $a \pref b \coloneqq a \bjoin \A(a) \compo b$. And there is an inverse interpretation of any representable $\{\compo, \A, \R, \pref\}$-algebra as a representable $\{\compo, \A, \R\}$-algebra with compatible joins, since for compatible $a$ and $b$ we have $a \bjoin b = a \pref b$. It remains to see that the morphisms are the same. Since $\{\compo, \A, \R\}$-homomorphisms preserve binary joins, they must preserve $\pref$, since $\pref$ is then definable in terms of preserved operations.
\end{proof}

Let $\Sigma$ be a finite alphabet. The \defn{rational} functions are the partial functions from $\Sigma^*$ to $\Sigma^*$ realisable by a one-way transducer. The \defn{regular} functions are the partial functions from $\Sigma^*$ to $\Sigma^*$ realisable by a two-way transducer. (See \cite{10.1145/2984450.2984453} for definitions of the various types of transducers.) The rational and the regular functions are both closed under composition, antidomain, and range and also under the partial operation of compatible union. These classes of partial functions are \emph{not} closed under other familiar operations that we may be tempted to include in the signature, such as intersection and relative complement. This is the reason for our interest in the  class $\class A$.

We mention one other important class of partial functions important to the theory of transducers. The \defn{sequential} functions are those partial functions realised by one-way \emph{input-deterministic} transducers. However, the sequential functions do not fit within our framework for the reason that they are not closed under compatible unions. (For example, the sequential functions $a^n \mapsto a^n$ and $a^n b \mapsto b^n$ have disjoint domains, hence are compatible, but their union is not sequential.)

\medskip

In view of \Cref{isomorphic}, we can choose to work with the representable $\{\compo, \A, \R, \pref\}$-algebras, in lieu of $\class A$, and henceforth that is what we will do. This pays off immediately: the class has a syntactically simple finite axiomatisation (and therefore is algebraically well behaved).

\begin{theorem}[{Hirsch, Jackson, and Mikul\'as \cite[Corollary~4.2 + Lemma~3.6]{hirsch}}]\label{axiomatisation}
The representable $\{\compo, \A, \R, \pref\}$-algebras form a proper quasivariety, axiomatised by the following finite list of equations and quasiequations.
\begin{align}
&&&&a \compo (b \compo c) &= (a \compo b) \compo c\\
&&&&\A(a) \compo a &= \A(b) \compo b\\
&&&&{\id} \compo a &= a\\
&&&&a \compo \A(b) &= \A(a \compo b) \compo a\\
\D(a) \compo b = \D(a) \compo c \;\;\land\;\; \A(a) \compo b &= \A(a) \compo c &\limplies&& b &= c\\
&&&&\D(\R(a)) &= \R(a)\\
&&&&a \compo \R(a) &= a\\
\label{quasi_range}a \compo b &= a \compo c &\limplies&& \R(a) \compo b &= \R(a) \compo c\\
&&&&\D(a) \compo (a \pref b) &= a\\
&&&&\A(a) \compo (a \pref b) &= \A(a) \compo b
\end{align}
\end{theorem} 

The category of representable $\{\compo, \A, \R, \pref\}$-algebras and their homomorphisms is the first of the two categories between which we will exhibit a duality. 

We now introduce a small running example by starting with an eight-element $\{\compo, \A, \R, \pref\}$-algebra. Though finite algebras cannot inform us greatly about the topological aspect of our duality---their duals all have discrete topologies---the example will be sufficient to grasp the essence of the duality. 

\begin{example}\label{small}
Let $\algebra A$ be the following collection of partial functions on the set $\{1, 2, 3\}$. The empty function, $\emptyset$, the identity $\operatorname{id}_{\{1,2\}}$ on $\{1, 2\}$, the identity $\operatorname{id}_{\{3\}}$ on $\{3\}$, the identity $\operatorname{id}_{\{1,2,3\}}$ on $\{1,2,3\}$, the `swap' $s \coloneqq \{1 \mapsto 2,\ 2\mapsto 1\}$, the function $\{1 \mapsto 2,\ 2\mapsto 1,\ 3\mapsto3\}$, the constant function $c \coloneqq \{1 \mapsto 3,\ 2\mapsto 3\}$, and the constant function $\{1 \mapsto 3,\ 2\mapsto 3,\ 3\mapsto3\}$. Then one can check that $\algebra A$ is closed under the operations of composition, antidomain, range, and preferential union and is therefore a $\{\compo, \A, \R, \pref\}$-algebra of partial functions.
\end{example}

Later we will take a particular interest in homomorphisms that are what we call \emph{locally proper}, though they are not essential to our duality. To define locally proper homomorphisms, we first need the notion of a prime filter in a representable $\{\compo, \A, \R, \pref\}$-algebra.

\begin{definition}
A homomorphism of representable $\{\compo, \A, \R, \pref\}$-algebras is \defn{locally proper} if the inverse image of every prime filter (see \Cref{basic_filter}) is a prime filter.
\end{definition}

This is the condition restricting the morphisms in Lawson's inverse semigroup duality \cite{Lawson_2010}. It is evident, even before reading the definition of a prime filter, that locally proper homomorphisms are closed under composition and include all identity maps.

\section{Stone \'etale categories}\label{category_definitions}

In this section we describe the other (large) category participating in our duality. It is a category of small categories with extra structure.\footnote{Unlike algebras, categories are allowed to be empty.} To reduce the potential for confusion, we will call the `object level' morphisms of the small categories \emph{arrows} (which underlines their abstract nature) and reserve \emph{morphism} for the `meta level' morphisms of the large category. Composition in the small categories is denoted ${\cdot}$ and like the $\compo$ notation, the first composee appears on the left-hand side.

\begin{definition}
A \defn{topological category}\footnote{Not to be confused with the various other (unrelated) usages of this term.} is a (small) category whose sets of objects $O$ and arrows $M$ are both topological spaces and such that 
\begin{itemize}
\item
the source map $\dd \from M \to O$ is continuous,
\item
the target map $\rr \from M \to O$ is continuous,
\item
the composition map ${\;\cdot\;} \from M \times_O M \to M$ is continuous,\footnote{The topology on the pullback $M \times_O M$ is the initial topology with respect to the two projections. That is, the topology generated by sets of the form $\{(x, y) \in M \times_O M \mid x \in U\}$ and $\{(x, y) \in M \times_O M \mid y \in U\}$ for open subsets $U$ of $M$. In other words, it is the subspace topology on $M \times_O M \subseteq M \times M$.}
\item
the identity-assigning map $x \mapsto 1_x$ sending each object to its identity arrow is continuous.
\end{itemize}
\end{definition}

Put concisely, for us a topological category is a category internal to the category $\normalfont\textbf{Top}$ of topological spaces. (See \cite[\S XII.1]{maclane:98} for the definition of internal categories.) Note that (from the `arrows only' viewpoint, \cite[p.~9]{maclane:98}) a topological category is a particular type of \emph{topological partial algebra}---a partial algebra on a topological space whose (possibly) partial operations are continuous when considered as functions on their domains of definition (equipped with the subspace topology).

\begin{definition}
A \defn{local homeomorphism} $\pi \from X \to Y$ of topological spaces is a continuous map such that for every $x \in X$ there exists an open neighbourhood $U$ of $x$ such that 
\begin{itemize}
\item
$\pi(U)$ is open,

\item
$\pi|_U \from U \to \pi(U)$ is a homeomorphism.
\end{itemize}
\end{definition}

\begin{definition}
An \defn{\'etale category} is a topological category such that 
\begin{itemize}
\item
the source map $\dd \from M \to O$ is a local homeomorphism,

\item
the target map $\rr \from M \to O$ is an open map.
\end{itemize}
An \'etale category is \defn{Stone} if its space of objects is a Stone space (also known as a Boolean space), that is, a compact and totally separated space.
\end{definition}

The condition that $\dd$ is a local homeomorphism and the condition, coming from the definition of a category, that $\dd$ is surjective, together say that in an \'etale category, $\dd$ gives $M$ the structure of an \emph{\'etale space} (of sets) over $O$ (also known as a \emph{sheaf space}).

One might expect to take functors given by continuous maps as the morphisms of topological categories. However we require a more general definition in order to capture all the duals of algebra homomorphisms.

\begin{definition}\label{multi-functor}
Let $\topo C$ and $\topo D$ be categories. A \defn{multivalued functor} $F \from \topo C \to \topo D$ consists of:
\begin{itemize}
\item
a function from the objects of $\topo C$ to the objects of $\topo D$, 
\item
a relation from the arrows of $\topo C$ to the arrows of $\topo D$,
\end{itemize}
(both denoted $F$) such that:
\begin{enumerate}
\item\label{one}
if $f \from x \to y$ is an arrow of $\topo C$ and $g \in F(f)$, then $g \from F(x) \to F(y)$,
\item
$1_{F(x)} \in F(1_x)$ for each object $x$ of $\topo C$,
\item\label{three}
if $f_1 \cdot f_2$ is defined, $g_1 \in F(f_1)$, and $g_2 \in F(f_2)$, for arrows $f_1$ and $f_2$, then  $g_1 \cdot g_2 \in F(f_1 \cdot f_2)$.
\end{enumerate}
For $F$ to be a multivalued functor between \emph{topological} categories we also require that:
\begin{enumerate}[(i)]
\item
the object component of $F$ is continuous,
\item
the arrow component of $F$ is a continuous relation from the arrows of $\topo C$ to the arrows of $\topo D$ (that is, inverse images of open sets are open).
\end{enumerate}
\end{definition}

Note that we are using `multivalued' in the sense `zero or more values'.
We now pick out certain special multivalued functors to account for the structure of the algebras for whose homomorphisms they are to provide duals.

\begin{definition}
A multivalued functor $F \from \topo C \to \topo D$ between categories is \defn{star injective} if for every object $x$ of $\topo C$, it restricts to an injective relation on the `star' $\operatorname{Hom}(x, -)$. That is:
\[F(f_1 \from x \to y_1) \cap F(f_2 \from x \to y_2) \neq \emptyset \implies f_1 = f_2.\] The same functor is \defn{star surjective} if its restrictions to stars are surjective relations.

We call a multivalued functor $F \from \topo C \to \topo D$ between topological categories \defn{pseudo star surjective} if whenever $U$ is an open set of arrows of $\topo D$, and $U \cap \operatorname{Hom}(F(x), -)$ is nonempty, then there exists some $f' \in U$ 
that is in the image of $F|_{\operatorname{Hom}(x, -)}$. The same functor is \defn{co-pseudo star surjective} if $F^{\operatorname{op}}$ is pseudo star surjective.

We call a multivalued functor between topological categories \defn{star coherent} if it is star injective, star surjective and co-pseudo star surjective.
\end{definition}

Composing two multivalued functors in the evident way yields a multivalued functor. The identity functor provides a two-sided identity for this composition.
 
\begin{lemma3}
Composition of multivalued functors between topological categories preserves star coherency.
\end{lemma3}

\begin{proof}
We give the details showing co-pseudo star surjectivity is preserved. Let $F \from \topo C \to \topo D$ and $G \from \topo D \to \topo E$ be co-pseudo star surjective multivalued functors. Suppose $U$ is open in $\topo E$, and $f \in U \cap \operatorname{Hom}(-, G(F(y))$. Then by co-pseudo star surjectivity of $G$, we can find some $g' \in G^{-1}(U) \cap \operatorname{Hom}(-, F(y))$. As $G$ is continuous, $G^{-1}(U)$ is open. Hence by co-pseudo star surjectivity of $F$, we can find an $h'' \in F^{-1}(G^{-1}(U)) \cap \operatorname{Hom}(-, y)$. That is, there is some $f'' \in U$ in the image of $(G \circ F)|_{\operatorname{Hom}(-, y)}$. Hence $G \circ F$ is co-pseudo star surjective.
\end{proof}

We see therefore that star-coherent multivalued functors can legitimately be used as morphisms of topological categories.

We are now finally ready to state our duality theorem.

\begin{theorem}\label{main-theorem}
There is a categorical duality between the following two categories.
\begin{itemize}
\item
The category $\mathscr A$ with 
\begin{description}
\item[objects] the $\{\compo, \A, \R, \pref\}$-algebras representable by partial functions,

\item[morphisms] the homomorphisms of $\{\compo, \A, \R, \pref\}$-algebras.
\end{description}

\item
The category $\mathscr C$ with
\begin{description}
\item[objects] the Stone \'etale categories all of whose arrows are epimorphisms,

\item[morphisms] the star-coherent multivalued functors of topological categories.
\end{description}
\end{itemize}
\end{theorem}

We will also show that the duality restricts to the following sub-duality.

\begin{theorem}\label{restricted}
There is a categorical duality between the following two categories.
\begin{itemize}
\item
The category $\mathscr A'$ with 
\begin{description}
\item[objects] the $\{\compo, \A, \R, \pref\}$-algebras representable by partial functions,

\item[morphisms] the locally proper homomorphisms of $\{\compo, \A, \R, \pref\}$-algebras.
\end{description}

\item
The category $\mathscr C'$ with
\begin{description}
\item[objects] the Stone \'etale categories all of whose arrows are epimorphisms,

\item[morphisms] the star-coherent functors of topological categories.
\end{description}
\end{itemize}
\end{theorem}

\section{From algebras to topological categories}\label{algebra_to_category}

In this section, we will define a contravariant functor in the direction from algebras to topological categories that forms one half of our duality. Following \cite{Lawson_2010}, we present the functor in terms of certain filters. That is, given an algebra, the entities used to construct a topological category---the entities that will constitute the arrows---will be filters satisfying a primality condition. We mention, however, that an alternative presentation is possible using the sort of algebraic distillation of \emph{germs of functions} found in \cite{Bauer_2013}, for example.

We start with some easily verifiable remarks. In a representable $\{\compo, \A, \R, \pref\}$-algebra $\algebra A$, the elements of the form $\A(-)$ form a subalgebra. We call an element of this subalgebra a \defn{domain element} (since it is equivalently of the form $\D(-)$). This sub\-algebra, $\D[\algebra A]$, is a Boolean algebra, with least element $0$, greatest element $\id$, meet given by $\compo$, complement given by $\A$, and join given either by De Morgan or by $\pref$.

In all the following lemmas, we will be working exclusively with isomorphism-invariant properties of representable $\{\compo, \A, \R, \pref\}$-algebras. Hence, whenever convenient, we may assume we are working genuinely with partial functions and make free use of any property of partial functions that is both intuitively obvious and easily verifiable from definitions. 

\begin{definition4}\label{basic_filter}
Let $\algebra A$ be a representable $\{\compo, \A, \R, \pref\}$-algebra. A \defn{filter} of $\algebra A$ is a nonempty, upward-closed, downward-directed subset of $\algebra A$. A filter $F$ is \defn{prime} if it is proper and whenever $a \pref b \in F$, either $a \in F$ or $b \in F$.\footnote{This is equivalent to the condition that whenever $a \bjoin b \in F$, either $a \in F$ or $b \in F$.}
\end{definition4}

Our definition of filters being the standard one, many basic facts are already known to us. For example there is a smallest filter including any given subset; that is, the notion of the filter \emph{generated} by a subset is well defined. Many of the properties of prime filters that we need have been proven in \cite[Section~4]{hirsch} (where a prime filter is called an \emph{ultrasubset}). The proofs there apply to any representable $\{\compo, \A, \R\}$-algebra, so in particular to the representable $\{\compo, \A, \R, \pref\}$-algebras. 

\begin{lemma4}\label{maximal}
In representable $\{\compo, \A, \R, \pref\}$-algebras, the prime and maximal filters coincide.\footnote{As is conventional, `maximal filter' will always mean maximal \emph{proper} filter.}
\end{lemma4}

\begin{proof}
Take first a prime filter $P$. Let $a$ be an arbitrary element not in $P$. Since $P$ is nonempty, we can find $p \in P$. The filter generated by $P \cup \{a\}$ must contain a lower bound for $\{a, p\}$, call this $b$. Since $b \leq p$ it follows, by reasoning about partial functions, that $p = b \bjoin (\A(b) \compo p)$. Hence either $b \in P$ or $\A(b) \compo p \in P$. But $b$ is not in $P$, else $a$ would be in $P$. Hence $\A(b) \compo p \in P$, and so the filter generated by $P \cup \{a\}$ contains both $b$ and $\A(b) \compo p$ and hence some lower bound for this pair---necessarily $0$. Hence any extension of $P$ is improper, so $P$ is maximal.

For the converse, we first establish:
\begin{equation*}\label{fact}\text{for any upward-closed set $S$, the set $\D[S]$ is upward closed in $\D[\algebra A]$}.\end{equation*}
 If $a \in S$ and $\D(b) \geq \D(a)$, then $a \pref b \geq a$, so $a \pref b \in S$. It is a property of partial functions that $\D(b) \geq \D(a) \implies \D(a \pref b) = \D(b)$, hence $\D[S]$ contains $\D(b)$. We conclude that $\D[S]$ is upward closed.

 Now take a maximal filter $M$, and suppose $a \pref b \in M$. By \cite[Lemma~4.5(ii)]{hirsch} and the fact that $\D[M]$ is upward closed in $\D[\algebra A]$, the set $\D[M]$ is an ultrafilter of $\D[\algebra A]$. Then either $\D(a) \in \D[M]$ or $\A(a) \in \D[M]$. Suppose the former, that is, there is some $c \in M$ with $\D(c) = \D(a)$. As $M$ is downward directed, there is some $d \in M$ with $d \leq a\pref b, c$. It follows, by reasoning about partial functions, that $d \leq a$, and hence $a \in M$. By a similar argument, if $\A(a) \in \D[M]$ then $b \in M$. Hence $M$ is prime.
\end{proof}

By \cite[Lemma~4.5(ii)]{hirsch} we now know that the following conditions are equivalent.
\begin{enumerate}
\item
$P$ is a prime filter.

\item
$P$ is a maximal filter.

\item
$P = (\mu \compo a)^\uparrow$ for some $a \in \algebra A$ and ultrafilter $\mu$ of $\D[\algebra A]$ such that $0 \not\in \mu \compo a$.

\item
For some ultrafilter $\mu$ of $\D[\algebra A]$, for all $a \in P$, we have $P = (\mu \compo a)^\uparrow$.
\end{enumerate}

In the following lemmas, let $\algebra A$ be a representable $\{\compo, \A, \R, \pref\}$-algebra. The notation $S^\uparrow$ denotes the upward closure of the set $S$ in $\algebra A$ or in $\D[\algebra A]$ if specified.

\begin{lemma4}\label{down}
Let $P$ be a prime filter of $\algebra A$. Then $\D[P]$ and $\R[P]^\uparrow$ are both ultrafilters of $\D[\algebra A]$, where the upward closure is taken in $\D[\algebra A]$.
\end{lemma4}

\begin{proof}
Since $P$ is nonempty, $\D[P]$ is too. As noted in the proof of \Cref{maximal}, the fact that $P$ is upward closed implies $\D[P]$ is too. Now suppose $\D(a), \D(b) \in \D[P]$, with $a, b \in P$. Then as $P$ is downward directed, it contains some $c \leq a, b$. This inequality implies (for any partial functions) that $\D(c) \leq \D(a), \D(b)$. Hence $\D[P]$ is downward directed. If $\D[P]$ contained $0$, then $P$ would have to too, since $\D(a) = 0$ implies $a = 0$ for partial functions. So $\D[P]$ must be a \emph{proper} filter. It remains to show that for any $a \in \algebra A$, either $\D(a)$ or $\A(a)$ belongs to $\D[P]$. Take any element $b$ of $P$. Then $b = (\D(a) \compo b) \pref (\A(a) \compo b)$, so either $\D(a) \compo b \in P$ or $\A(a) \compo b \in P$. Then we can obtain an element of $\D[P]$ that is less than or equal to either $\D(a)$ or $\A(a)$, respectively.

As $P$ is nonempty, $\R[P]^\uparrow$ is too. It is upward closed by definition. Suppose $\R(a), \R(b) \in \R[P]$, with $a, b \in P$. Then as $P$ is downward directed, it contains some $c \leq a, b$. The inequality implies  that $\R(c) \leq \R(a), \R(b)$. Hence $\R[P]$ is downward directed, so $\R[P]^\uparrow$ is too. Since  $\R[P]$ cannot contain $0$, neither can $\R[P]^\uparrow$, so it is proper. Given any $a \in \algebra A$, take any $b \in P$. Then $b =  (b \compo \D(a)) \pref (b \compo \A(a))$, so either $b \compo \D(a) \in P$ or $b \compo \A(a) \in P$. But $\R(b \compo \D(a)) \leq \D(a)$ and $\R(b \compo \A(a)) \leq \A(a)$. So either $\D(a) \in \R[P]^\uparrow$ or $\A(a) \in \R[P]^\uparrow$, respectively.
\end{proof}

\begin{lemma4}\label{up}
Let $\mu$ be an ultrafilter of $\D[\algebra A]$. Then $\mu^\uparrow$, where the upward closure is taken in $\algebra A$, is a prime filter of $\algebra A$.
\end{lemma4}

\begin{proof}
We have $\mu^\uparrow = (\mu \compo \id)^\uparrow$, and by \cite[Lemma~4.5(ii)]{hirsch} this is a maximal filter. By \Cref{maximal}, it is a prime filter.
\end{proof}

\begin{lemma4}\label{g}
Let $\mu$ be an ultrafilter of $\D[\algebra A]$. Then $\D[\mu^\uparrow] = \mu$, where the upward closure is taken in $\algebra A$.
\end{lemma4}

\begin{proof}
We know $\mu \subseteq \D[\mu^\uparrow]$ because $\D$ fixes all domain elements. Conversely, take an element of $\D[\mu^\uparrow]$: an element $\D(b)$ such that $b \geq \alpha$ for some $\alpha \in \mu$. Then $\D(b) \geq \D(\alpha)$. But $\D(\alpha) = \alpha$ and $\mu$ is upward closed; hence $\D(b) \in \mu$. This proves the reverse inclusion.
\end{proof}

\begin{lemma4}\label{composition}
Let $P, Q$ be prime filters of $\algebra A$. Then $(P \compo Q)^\uparrow$ is a prime filter if and only if $\D[Q] =\R[P]^\uparrow$. Otherwise $(P \compo Q)^\uparrow$ contains $0$ and hence is all of $\algebra A$.
\end{lemma4}



\begin{proof}
Let $P = (\mu \compo a)^\uparrow$ and $Q = (\nu \compo b)^\uparrow$ for $\mu, \nu$ ultrafilters of domain elements. By \cite[Lemma~4.6(i)]{hirsch}, the set $(P \compo Q)^\uparrow$ is of the form $(\mu \compo a \compo \nu \compo b)^\uparrow$. By \cite[Lemma~4.6(ii)]{hirsch}, the set $(\mu \compo a \compo \nu \compo b)^\uparrow$ is a prime filter if and only if it does not contain $0$. So it remains to show that $0 \in (P \compo Q)^\uparrow$ if and only if $\D[Q] \neq\R[P]^\uparrow$. 

First suppose $0 \in (P \compo Q)^\uparrow$, so $0 \in P \compo Q$, and so $0 = a \compo b$ for some $a \in P$ and $b \in Q$. It follows, by reasoning about partial functions, that $\R(a) \compo \D(b) = 0$. Since $\R(a) \in \R[P]^\uparrow$ and $\D(b) \in \D[Q]$, and $\compo$ is meet on domain elements, these two ultrafilters cannot be equal. Conversely, suppose $\D[Q] \neq\R[P]^\uparrow$. Since they are ultrafilters, we can then find Boolean complements $\alpha$ and $\A(\alpha)$ with $\alpha \in \R[P]^\uparrow$ and $\A(\alpha) \in \D[Q]$. That is, there are $a \in P$ with $\R(a) \leq \alpha$ and $b \in Q$ with $\D(b) = \A(\alpha)$. Then $\R(a) \compo \D(b) = 0$, and it follows by reasoning about partial functions that $a \compo b = 0$; hence $0 \in (P \compo Q)^\uparrow$.
\end{proof}

\begin{lemma4}\label{right_restriction}
Let $P, Q$ be prime filters of $\algebra A$. If $(P \compo Q)^\uparrow$ is proper, then $\R[P \compo Q]^\uparrow = \R[Q]^\uparrow$.
\end{lemma4}

\begin{proof}
First we show  $\R[P \compo Q] \subseteq \R[Q]$, giving $\R[P \compo Q]^\uparrow \subseteq \R[Q]^\uparrow$. Take an $\R(a \compo b) \in \R[P \compo Q]$, with $a \in P$ and $b \in Q$. Now $b = (\R(a) \compo b) \pref (\A(\R(a)) \compo b)$, so, since $Q$ is a prime filter, either it contains  $\R(a) \compo b$ or $\A(\R(a)) \compo b$. If the latter, then $P \compo Q$ contains $a \compo \A(\R(a)) \compo b$, equal to $0$, contradicting the hypothesis that $(P \compo Q)^\uparrow$ is proper. Hence $Q$ contains $\R(a) \compo b$, so $\R[Q]$ contains $\R(\R(a) \compo b)$. But $\R(\R(a) \compo b) = \R(a \compo b)$ is a property of partial functions (it  is axiom (r.VII) in \cite{hirsch}). Hence $\R(a \compo b) \in \R[Q]$, and we have our first inclusion.

Conversely, suppose $\R(b) \in \R[Q]$, with $b \in Q$. Take any $a \in P$. As before, it must be the case that $\R(a) \compo b \in Q$. Then $a \compo (\R(a) \compo b) \in P \compo Q$, that is, $a \compo b \in P \compo Q$. Hence $\R(a \compo b) \in \R[P \compo Q]$. By a property of partial functions, $\R(b) \geq \R(a \compo b)$, hence $\R(b) \in \R[P \compo Q]^\uparrow$. Since $\R(b)$ was an arbitrary element of $\R[Q]$ we have $\R[Q] \subseteq \R[P \compo Q]^\uparrow$ and hence $\R[Q]^\uparrow \subseteq \R[P \compo Q]^\uparrow$.
\end{proof}

\begin{lemma4}\label{intersection}
If $P$ and $Q$ are nondisjoint prime filters with $\D[P] = \D[Q]$, then $P = Q$.
\end{lemma4}

\begin{proof}
This is \cite[Lemma~4.5(iv)]{hirsch}.
\end{proof}


\begin{lemma4}\label{cancellation}
Let $P, Q, R$ be prime filters of $\algebra A$. Suppose $(P \compo Q)^\uparrow$ and $(P \compo R)^\uparrow$ are proper and equal. Then $Q = R$.
\end{lemma4}

\begin{proof}
If $(P \compo Q)^\uparrow$ and $(P \compo R)^\uparrow$ are proper, then by \Cref{composition}, we have $\D[Q] = \R[P]^\uparrow = \D[R]$. So by \Cref{intersection} it is sufficient to show that $Q$ and $R$ are nondisjoint. Let $a \in P$ and $b \in Q$. So $a \compo b \in (P \compo Q)^\uparrow = (P \compo R)^\uparrow$, that is, $a \compo b \geq a' \compo c$ for some $a' \in P$ and $c \in R$. By definition this means $\D(a' \compo c) \compo a \compo b = a' \compo c$. But $\D(a' \compo c)$ must belong to the ultrafilter $\D[P]$---it cannot be that its Boolean complement $\A(a' \compo c)$ is in $\D[P]$, since $(P \compo R)^\uparrow$ is proper. Hence $\D(a' \compo c) \compo a \in P$. Pick some $a'' \in P$ with $a'' \leq \D(a' \compo c) \compo a$ and $a'' \leq a'$. It is a property of partial functions that if $d \compo b = d' \compo c$ and $d'' \leq d, d'$ then $d'' \compo b = d'' \compo c$. Hence, from $\D(a' \compo c) \compo a \compo b = a' \compo c$ we obtain $a'' \compo b = a'' \compo c$. By axiom \eqref{quasi_range}, this gives $\R(a'') \compo b = \R(a'') \compo c$. Now $\R(a'')$ is an element of $\D[Q] = \R[P]^\uparrow = \D[R]$, so $\R(a'') \compo b \in Q$ and $\R(a'') \compo c \in R$. That is, we have found our element common to $Q$ and $R$.
\end{proof}

\begin{lemma4}\label{filter-range}
Let $\mu$ be an ultrafilter of $\D[\algebra A]$ and $a \in \algebra A$ be such that $\R(a) \in \mu$. Then there exists a prime filter $P$ such that $a \in P$ and $\R[P]^\uparrow = \mu$.
\end{lemma4}

\begin{proof}
Suppose $\R(a) \in \mu$. Consider the subset $\D[a \compo \mu]$ of the Boolean algebra $\D[\algebra A]$. This set $\D[a \compo \mu]$ is nonempty (because $\mu$ is nonempty) and downward directed---because given $\D[a \compo \alpha]$ and $\D[a \compo \beta]$, for $\alpha, \beta \in \mu$, we know $\alpha \compo \beta \in \mu$, and it is a property of partial functions that $\D(a \compo \alpha \compo \beta)$ is a lower bound for $\{\D(a \compo \alpha), \D(a \compo \beta)\}$ (in fact it is the meet). Further, $\D[a \compo \mu]$ does not contain $0$, since that would imply $0 \in a \compo \mu$, which implies that $\A(\R(a)) \in \mu$, but this is prohibited, since $\mu$ is an ultrafilter containing $\R(a)$.

 We have shown that $\D[a \compo \mu]^\uparrow$ is a proper filter of $\D[\algebra A]$. Extend it to an ultrafilter $\nu$. Now $(\nu \compo a)^\uparrow$ is a prime filter because $0 \not\in \nu \compo a$, by the following reasoning. The filter $\D[a \compo \mu]^\uparrow$ contains $\D(a \compo \id) = \D(a)$, and hence $\nu$ does too, meaning $\nu$ does not contain $\A(a)$, which is a necessary condition for $\nu \compo a$ to contain $0$. Our prime filter $(\nu \compo a)^\uparrow$ contains $a$, as desired.

 Finally, we claim that $\R[(\nu \compo a)^\uparrow]^\uparrow = \mu$. For any $\alpha \in \mu$, we know that $\D(a \compo \alpha) \in \nu$ and therefore $\R(\D(a \compo \alpha) \compo a) \in \R[(\nu \compo a)^\uparrow]^\uparrow$. But it is a property of partial functions that $\R(\D(a \compo \alpha) \compo a) \leq \alpha$ (given that $\alpha$ is a domain element). Since $\R[(\nu \compo a)^\uparrow]^\uparrow$ is upward closed and $\alpha$ was an arbitrary element of $\mu$, we obtain $\mu \subseteq \R[(\nu \compo a)^\uparrow]^\uparrow$. Since $\R[(\nu \compo a)^\uparrow]^\uparrow$ and $\mu$ are ultrafilters, they are equal, as required.
\end{proof}

\subsection{The functor \texorpdfstring{$\pf$}{PF} on objects}

We now define the functor $\pf \from \mathscr A \to \mathscr C$ used for one half of the duality. (The $\pf$ stands for `prime filter', not `partial function'!) For  $\algebra A$ a representable $\{\compo, \A, \R, \pref\}$-algebra, let $\pf( \algebra A)$ be the following Stone \'etale category.
\begin{itemize}
\item
 The objects of $\pf (\algebra A)$ are the ultrafilters of $\D[\algebra A]$. 
 
\item
The arrows of $\pf (\algebra A)$ are the prime filters of $\algebra A$.

\item
The source and target of an arrow $P$ are $\D[P]$ and $\R[P]^\uparrow$ respectively. By \Cref{down} these are objects. 

\item
The identity arrow for an object $\mu$ is $\mu^\uparrow$, where the upward closure is taken in $\algebra A$. By \Cref{up}, this is an arrow. By \Cref{g}, its source  is $\mu$. By \Cref{composition}, the target of $\mu^\uparrow$ must also be $\mu$ since $(\mu^\uparrow \compo \mu^\uparrow)^\uparrow$ is proper.

\item
For composable arrows $P$ and $Q$, the composition is given by $P \ccompo Q \coloneqq (P \compo Q)^\uparrow$. By \Cref{right_restriction}, this is an arrow, evidently with the same source as $P$. By \Cref{composition}, it has the same target as $Q$. 
\end{itemize}
The confirmation that the structure so defined validates the axioms for categories is the content of \Cref{axioms}, which follows shortly. By \Cref{cancellation}, all arrows are epimorphisms. Let $\operatorname{uf}(\D[\algebra A])$ denote the ultrafilters of $\D[\algebra A]$, and let $\operatorname{pf}(\algebra A)$ denote the prime filters of $\algebra A$.
\begin{itemize}

\item
The topology on the objects is the topology generated by $\{ \widehat \alpha \mid \alpha \in \D[\algebra A]\}$, where $\widehat \alpha \coloneqq \{\mu \in \operatorname{uf}(\D[\algebra A]) \mid \alpha \in \mu\}$. 

\item
The topology on the arrows is the topology generated by $\{ a^\theta \mid a \in \algebra A\}$, where $a^\theta \coloneqq \{P \in \operatorname{pf}(\algebra A) \mid a \in P\}$. 
\end{itemize}
The confirmations that the source, target, composition, and identity-assigning maps are continuous with respect to these topologies is the content of the following \Cref{continuous}. The confirmation that the source map is a local homeomorphism is the following \Cref{source_lemma}, and the confirmation that the target map is an open map is \Cref{target_lemma}. It is immediate that the objects form a Stone space, since we have used for this space exactly the standard construction of the Stone dual of the Boolean algebra $\D[\algebra A]$.

\begin{lemma4}\label{axioms}
Let $\algebra A$ be a representable $\{\compo, \A, \R, \pref\}$-algebra. Then $\pf (\algebra A)$ satisfies the associativity and identity axioms for categories.
\end{lemma4}

\begin{proof}
First the identity laws: let $P$ be an arrow (a prime filter) and let $\mu = \D[P]$ be its source, so $P$ is of the form $(\mu \compo a)^\uparrow$ for some $a$. Then the identity arrow at $\mu$ is $\mu^\uparrow = (\mu \compo \id)^\uparrow$ and contains $\id$. By definition $\mu^\uparrow \ccompo P = (\mu^\uparrow \compo P)^\uparrow$ and so contains ${\id} \compo a = a$. Then since $\D[\mu^\uparrow \ccompo P] = \D[\mu^\uparrow] = \mu = \D[P]$, by \Cref{intersection} we conclude $\mu^\uparrow \ccompo P = P$. Similarly, $P \ccompo \mu^\uparrow$ contains $a$ and has source equal to $\D[P]$. Hence we also have $P = P \ccompo \mu^\uparrow$.

For the associativity law, by similar reasoning, if the compositions $(P \ccompo Q) \ccompo R$ and $P \ccompo (Q \ccompo R)$ are defined, then they are nondisjoint. And since both have source $\D[P]$, by \Cref{intersection} they are equal.
\end{proof}

\begin{lemma4}\label{continuous}
Let $\algebra A$ be a representable $\{\compo, \A, \R, \pref\}$-algebra. The source, target, composition, and identity-assigning maps on the category $\pf(\algebra A)$ are continuous with respect to the topologies generated by $\{ \widehat \alpha \mid \alpha \in \D[\algebra A]\}$ and $\{ a^\theta \mid a \in \algebra A\}$.
\end{lemma4}

\begin{proof}
First $\dd$: we take $\widehat \alpha$ and show that $\dd^{-1}(\widehat \alpha)$ is open. So let $P$ be a prime filter with $\alpha \in \D[P]$. Take any $a \in P$. Then $b \coloneqq \alpha \compo a$ is also in $P$, so $P \in b^\theta$. Now for any $Q \in b^\theta$, we have $\D(b) \in \D[Q] = \dd(Q)$, and $\D(b) \leq \alpha$, so $\alpha \in \dd(Q)$. That is, $Q \in \dd^{-1}(\widehat \alpha)$. So $P \in b^\theta \subseteq \dd^{-1}(\widehat \alpha)$. Since $P$ was an arbitrary element of $\dd^{-1}(\widehat \alpha)$ and $b^\theta$ is by definition open, we are done.

Next $\rr$: we take $\widehat \alpha$ and show that $\rr^{-1}(\widehat \alpha)$ is open. So let $P$ be a prime filter with $\alpha \in \R[P]^\uparrow$. Take any $a \in P$. Then $b \coloneqq a \compo \alpha$ is also in $P$---because $a = (a \compo \alpha) \pref (a \compo \A(\alpha))$, but $a \compo \A(\alpha)$ cannot be in $P$, else $\R[P]^\uparrow$ (which contains $\alpha$) would contain $0$. Hence $P \in b^\theta$. Now for any $Q \in b^\theta$, we have $\R(b) \in \R[Q]^\uparrow = \rr(Q)$, and $\R(b) \leq \alpha$, so $\alpha \in \rr(Q)$. That is, $Q \in \rr^{-1}(\widehat \alpha)$. So $P \in b^\theta \subseteq \rr^{-1}(\widehat \alpha)$. Since $P$ was an arbitrary element of $\rr^{-1}(\widehat \alpha)$ and $b^\theta$ is by definition open, we are done.

For composition: we take $a^\theta$ and show that the inverse image under $\ccompo$ is open. So let $P$ and $Q$ be two prime filters such that $P \cdot Q \in a^\theta$, that is, $a \in (P \compo Q)^\uparrow$. Then there are $b \in P$ and $c \in Q$ such that $b \compo c \leq a$. Now $\D(b \compo c) \compo b$ must also be in $P$, for if $\A(b \compo c) \compo b$ were in $P$ then $P \compo Q$ would contain $\A(b \compo c) \compo b \compo c = 0$. So we have open sets $(\D(b \compo c) \compo b)^\theta$ containing $P$ and $c^\theta$ containing $Q$, and for any two composable arrows $R \in (\D(b \compo c) \compo b)^\theta$ and $S \in c^\theta$, their composition $R\ccompo S = (R \compo S)^\uparrow$ contains $\D(b \compo c) \compo b \compo c = b \compo c \leq a$ and therefore lies in $a^\theta$. Since $P$ and $Q$ were arbitrary subject to $P \ccompo Q \in a^\theta$, this proves that $a^\theta$ is open.

For the identity-assigning map: we take $a^\theta$ and confirm that the set $\{\mu \in \operatorname{uf}(\D[\algebra A]) \mid a \in \mu^\uparrow\}$ is open. Take $\nu$ in this set; then $a \geq \alpha$ for some $\alpha \in \nu$. It follows that the open set $\widehat\alpha$ of objects contains $\nu$ and is included in $\{\mu \in \operatorname{uf}(\D[\algebra A]) \mid a \in \mu^\uparrow\}$, so we are done.
\end{proof}

\begin{lemma4}\label{source_lemma}
The map $\dd \from P \mapsto \D[P]$ is a local homeomorphism from the arrows to the objects of $\pf(\algebra A)$.
\end{lemma4}

\begin{proof}
We know from \Cref{continuous} that $\dd$ is continuous. Next we establish that $\dd$ is an open map by showing that $\dd[a^\theta] = \widehat{\D(a)}$ for any $a \in \algebra A$. Clearly if $a$ belongs to a prime filter $P$ then $\D(a)$ belongs to $\dd(P) = \D[P]$, hence $\dd[a^\theta] \subseteq \widehat{\D(a)}$. Conversely, any $\mu \in \widehat{\D(a)}$ is the image under $\dd$ of the element $(\mu \compo a)^\uparrow$ of $a^\theta$ (for $\D(a) \in \mu$ ensures $0 \not\in \mu \compo a$). Hence $\dd$ is an open map.

Now any open and continuous map $f$ is a local homeomorphism if every point in its domain has an open neighbourhood $U$ such that the restriction of $f$ to $U$ is injective. For $\dd$ we take, for any $P$ in its domain, any $a \in P$ we wish and use the open neighbourhood $a^\theta$ of $P$. The map $\dd$ is injective on $a^\theta$ by \Cref{intersection}.
\end{proof}

\begin{lemma4}\label{target_lemma}
The map $\rr \from P \mapsto \R[P]^\uparrow$ is an open map from the arrows to the objects of $\pf(\algebra A)$.
\end{lemma4}

\begin{proof}
We argue that $\rr[a^\theta] = \widehat{\R(a)}$ for any $a \in \algebra A$. Clearly if $\mu \in \rr[a^\theta]$, that is, $\mu = \R[P]^\uparrow$ for some $P$ containing $a$, then $\mu$ contains $\R(a)$, so $\mu \in \widehat{\R(a)}$. Conversely, suppose $\mu \in \widehat{\R(a)}$, that is, $\R(a) \in \mu$. By \Cref{filter-range}, there exists a prime filter $P$ containing $a$ and such that $\R[P]^\uparrow = \mu$. That is, $P \in a^\theta$ and $\mu = \R[P]^\uparrow = \rr(P) \in \rr[a^\theta]$.
\end{proof}

\begin{example}\label{small_dual}
Let $\algebra A$ be the example from \Cref{small}. The Boolean subalgebra $\D[\algebra A]$ consists of the four elements $\emptyset$, $\operatorname{id}_{\{1,2\}}$, $\operatorname{id}_{\{3\}}$, and $\operatorname{id}_{\{1,2,3\}}$. The ultrafilters of $\D[\algebra A]$ (the objects of the dual) are $\{\operatorname{id}_{\{1,2\}}, \operatorname{id}_{\{1,2,3\}}\}$ and $\{\operatorname{id}_{\{3\}}, \operatorname{id}_{\{1,2,3\}}\}$, which we call `$1,2$' and `$3$' respectively. The prime filters of $\algebra A$ (the arrows of the dual) are the up-sets of the minimal nonzero elements of $\algebra A$, and there are four of these: $\operatorname{id}_{\{1,2\}}$, $\operatorname{id}_{\{3\}}$, $s$, and $c$. Those that correspond to \emph{identity} arrows in the dual are $\operatorname{id}_{\{1,2\}}$ and $\operatorname{id}_{\{3\}}$. We can calculate that both $\operatorname{id}_{\{1,2\}}$ and $s$ have source `$1,2$' and target `$1,2$', that $s \cdot s = \operatorname{id}_{\{1,2\}}$, and so on. A suggestive diagram of the dual $\pf(\algebra A)$ of $\algebra A$ follows. Note how the dual is smaller and easier to depict than the algebra.
\begin{figure}[H]\centering
\caption{The dual $\pf(\algebra A)$ of $\algebra A$}
\begin{tikzpicture}
   \node[state] (1) {$1, 2$}; 
   \node[state] (2) [right=of 1] {$3$}; 
    \draw[->]
     (1)    edge [loop left, double]  ()
    (2) 
         edge [loop right] ();      
    \draw[->]
    (1) edge [bend left, above] (2);
\end{tikzpicture}
\end{figure}
\end{example}

\subsection{The functor \texorpdfstring{$\pf$}{PF} on morphisms}

Let $h \from \algebra A \to \algebra B$ be a homomorphism of representable $\{\compo, \A, \R, \pref\}$-algebras. It is immediate that $h$ restricts to a Boolean algebra homomorphism from $\D[\algebra A]$ to $\D[\algebra B]$. The action of $\pf$ on morphisms of $\mathscr A$ is given by inverse image. In more detail, the continuous multivalued functor $\pf(h) \from \pf(\algebra B) \to \pf(\algebra A)$ is given by:
\begin{itemize}
\item
for an object $\mu \in \operatorname{uf}(\D[\algebra B])$:
\[\mu \mapsto (h|_{\D[\algebra A]})^{-1}(\mu),\]
\item
for an arrow $P \in \operatorname{pf}(\algebra B)$: \[P \mapsto \{Q \in \operatorname{pf}(\algebra A) \mid Q \subseteq h^{-1}(P)\}.\]
\end{itemize}
That the object component of $\pf(h)$ is a well defined and continuous function follows from classical Stone duality. The proof that $\pf(h)$ is a multivalued functor from the underlying category of $\pf(\algebra B)$ to that of $\pf(\algebra A)$ is \Cref{low_functor}. The proof that the arrow component of $\pf(h)$ is continuous is \Cref{continuous_functor}. The proof that $\pf(h)$ is star coherent is \Cref{star}. The proof that $\pf$ is itself functorial is \Cref{functorial}.

\begin{lemma4}\label{partition}
For any homomorphism $h \from \algebra A \to \algebra B$ and any prime filter $P$ of $\algebra B$, the set $h^{-1}(P)$ is partitioned into prime filters of $\algebra A$.
\end{lemma4}

\begin{proof}
Let $\mu$ be the ultrafilter $\D[P]$ and $\nu$ be the ultrafilter $ (h|_{\D[\algebra A]})^{-1}(\mu)$. The relation $a \sim b \iff \exists \alpha \in \nu : \alpha \compo a = \alpha \compo b$ is easily seen to be an equivalence relation on $h^{-1}(P)$. We claim that each $\sim$-equivalence class is a prime filter. 

Each equivalence class is by definition nonempty. Each equivalence class is a proper subset of $\algebra A$, for if $h^{-1}(P)$ contained $0$ then $P$ would contain $h(0) = 0$, in contradiction to $P$ being a prime filter.

We now argue equivalence classes are upward closed. Take $a \in h^{-1}(P)$ and $a' \geq a$. As $h$ is a homomorphism, it is order preserving, so $h^{-1}$ maps upward closed sets to upward closed sets. Hence, as $P$ is upward closed, $h^{-1}(P)$ is upward closed. So $a' \in h^{-1}(P)$. Further, $h(\D(a)) = \D(h(a)) \in \D[P] = \mu$, so $\D(a) \in \nu$. Hence $a \sim a'$, for $\D(a) \compo a' = a = \D(a) \compo a$. We conclude that $\sim$-equivalence classes are upward closed.

To show that equivalence classes are downward directed, take $a \sim b$ and $\alpha \in \nu$ such that $\alpha \compo a = \alpha \compo b$. By elementary reasoning about partial functions, $\alpha \compo a$ is a lower bound for $a$ and $b$, and \emph{if} $\alpha \compo a \in h^{-1}(P)$ then $\alpha \compo a \sim a$. Hence we only need to show $\alpha \compo a \in h^{-1}(P)$. We know $h(a) \in P$ and $h(\alpha) \in \mu$, and so $h(\alpha) \compo h(a) \in P$. But as $h$ is a homomorphism, $h(\alpha) \compo h(a) = h(\alpha \compo a)$, so $\alpha \compo a \in h^{-1}(P)$, as desired.

Finally, to show equivalence classes satisfy the primality condition, take $ b \pref c =a \in h^{-1}(P)$. Then as $\nu$ is an ultrafilter of $\D[\algebra A]$, either $\D(b) \in \nu$ or $\A(b) \in \nu$. If $\D(b) \in \nu$ then by the same argument appearing in the previous paragraph, $\D(b) \compo a \in h^{-1}(P)$. But $\D(b) \compo a = b$, so then $b \in h^{-1}(P)$. Similarly, if $\A(b) \in \nu$ then $\A(b) \compo a \in h^{-1}(P)$. But $\A(b) \compo a \leq c$ and $h^{-1}(P)$ is upward closed, so then $c \in h^{-1}(P)$.
\end{proof}

\begin{lemma4}\label{low_functor}
For any homomorphism $h \from \algebra A \to \algebra B$, the function/relation pair $\pf(h)$ is a multivalued functor between the underlying categories of $\pf(\algebra B)$ and $\pf(\algebra A)$.
\end{lemma4}

\begin{proof}
We must show that $\pf(h)$ validates conditions \ref{one}--\ref{three} of \Cref{multi-functor}.
\begin{enumerate}
\item
Let $P \from \D[P] \to \R[P]^\uparrow$ be an arrow in $\pf(\algebra B)$ (that is, a prime filter of $\algebra B$) and suppose $Q \subseteq \pf(h)(P)$ is a prime filter of $\algebra A$. We want to show that the ultrafilters $\D[P]$ and $\R[P]^\uparrow$ are mapped to the source and target of $Q$ respectively. That is, we want to show $(h|_{\D[\algebra A]})^{-1}(\D[P]) = \D[Q]$ and $(h|_{\D[\algebra A]})^{-1}(\R[P]^\uparrow) = \R[Q]^\uparrow$. We prove the second equality; the proof of the first is similar (but simpler). Suppose $\alpha \in \R[Q]^\uparrow$, so $\alpha \geq \beta = \R(b)$ for some $b \in Q$. Then $h(b) \in P$, so $h(\beta) = h(\R(b)) = \R(h(b)) \in \R[P]$. That is, $\beta \in (h|_{\D[\algebra A]})^{-1}(\R[P])$ (since $\beta \in \D[\algebra A]$). As $\beta \in (h|_{\D[\algebra A]})^{-1}(\R[P])$ is an ultrafilter, by upward closure $\alpha \in (h|_{\D[\algebra A]})^{-1}(\R[P])$ also. We have shown that $(h|_{\D[\algebra A]})^{-1}(\R[P]^\uparrow) \supseteq \R[Q]^\uparrow$. The reverse inclusion follows, since both sides are ultrafilters.

\item
We want to show that for every ultrafilter $\mu$ of $\D[\algebra B]$, the identity arrow on $\nu \coloneqq \pf(h)(\mu) = (h|_{\D[\algebra A]})^{-1}(\mu)$ belongs to $\pf(h)(\mu^\uparrow)$ (upward closure in $\algebra B$), that is, is a subset of $h^{-1}(\mu ^\uparrow)$. The identity arrow on $\nu$ is $\nu^\uparrow$. Suppose $a \geq \alpha \in \nu$. Then $h(a) \geq h(\alpha) \in \mu$, so $h(a) \in \mu^\uparrow$. That is, $a \in h^{-1}(\mu^\uparrow)$, as required.

\item
Suppose $P_1 \ccompo P_2$ is defined, $Q_1 \subseteq h^{-1}(P_1)$, and $Q_2 \subseteq h^{-1}(P_2)$. We know $Q_1 \ccompo Q_2$ is defined and has the same source as the prime filters that, by \Cref{partition}, partition $h^{-1}(P_1 \ccompo P_2)$. Choose some $a \in Q_1$ and $b \in Q_2$. Then $a \compo b \in Q_1 \cdot Q_2$, and $h(a \compo b) = h(a) \compo h(b) \in P_1 \ccompo P_2$. So $a \compo b$ also belongs to $h^{-1}(P_1 \ccompo P_2)$. Hence the prime filter $Q_1 \ccompo Q_2$ has nonempty intersection with one of the prime filters partitioning $h^{-1}(P_1 \ccompo P_2)$. By \Cref{intersection}, $Q_1 \ccompo Q_2$ \emph{equals} that prime filter. So $Q_1 \ccompo Q_2 \subseteq h^{-1}(P_1 \ccompo P_2)$, that is, $Q_1 \ccompo Q_2 \in \pf(h)(P_1 \ccompo P_2)$, as required.\qedhere
\end{enumerate}
\end{proof}

\begin{lemma4}\label{continuous_functor}
For any homomorphism $h \from \algebra A \to \algebra B$, the arrow component of the multivalued functor $\pf(h) \from\allowbreak \pf(\algebra B) \to \pf(\algebra A)$ is continuous.
\end{lemma4}

\begin{proof}
Since sets of the form $a^\theta$ form a subbasis for the topology on $\algebra A$, it suffices to show that given $a_1, \dots, a_n \in \algebra A$, the set $\pf(h)^{-1}(\bigcap_i a_i^\theta)$ is open in $\algebra B$. Suppose $P \in \pf(h)^{-1}(\bigcap_i a_i^\theta)$. Then there exists a prime filter $Q \subseteq h^{-1}(P)$ with $Q \in \bigcap_i a_i^\theta$. So $a_1, \dots, a_n \in Q$. Since $Q$ is a filter, there is an $a \in Q$ with $a \leq a_1, \dots, a_n$. Then $h(a) \in P$, that is, $P \in h(a)^\theta$. And for any prime filter $P'$ of $\algebra B$:
\begin{align*}
P' \in h(a)^\theta &\implies h(a) \in P'\\
&\implies a \in h^{-1}(P')\\
&\implies \exists {Q' \in \operatorname{pf(\algebra A)}} : a \in Q' \subseteq h^{-1}(P'),
\end{align*}
so for such a $Q'$:
\begin{align*}
\phantom{P' \in h(a)^\theta}&\implies a_1, \dots, a_n \in Q'\\
&\implies Q' \in a_1^\theta, \dots, a_n^\theta\\
&\implies Q' \in \bigcap_i a_i^\theta,
\end{align*}
and hence $P' \in \pf(h)^{-1}(\bigcap_i a_i^\theta)$. So the open set $h(a)^\theta$ contains $P$ and is included entirely within $\pf(h)^{-1}(\bigcap_i a_i^\theta)$. Hence $\pf(h)^{-1}(\bigcap_i a_i^\theta)$ is open.
\end{proof}

\begin{lemma4}\label{star}
For any homomorphism $h \from \algebra A \to \algebra B$, the multivalued functor $\pf(h) \from\allowbreak \pf(\algebra B) \to \pf(\algebra A)$ is star coherent.
\end{lemma4}

\begin{proof}
For star injectivity, suppose $P_1, P_2 \in \pf(\algebra B)$ are prime filters with the same source---$\D[P_1] = \D[P_2]$---and such that $Q \in \pf(h)(P_1) \cap \pf(h)(P_2)$. From the latter, we have $Q \subseteq h^{-1}(P_1)$ and $Q \subseteq h^{-1}(P_2)$. Then as $Q$ is nonempty, we can choose some $a \in Q$ and obtain $h(a) \in P_1, P_2$. By \Cref{intersection}, we get $P_1 = P_2$, confirming star injectivity.

For star surjectivity, suppose $\mu$ is an ultrafilter of $\D[\algebra B]$, and $Q$ is a prime filter of $\algebra A$ with $\D[Q] = (h|_{\D[\algebra A]})^{-1}(\mu)$. Our objective is to find some prime filter $P$ of $\algebra B$ with $\D[P] = \mu$ and $Q \in \pf(h)(P)$. Choose some element $a \in Q$. Then $\D(a) \in h^{-1}(\mu)$, so $h(\D(a)) \in \mu$. Since $h$ is a homomorphism, this gives $\D(h(a)) \in \mu$. This implies $0 \not\in (\mu \compo h(a))^\uparrow$, and hence $(\mu \compo h(a))^\uparrow$ is a prime filter of $\algebra B$, which we denote by $P$. We know $\D[P] = \mu$, and we claim that $Q \in \pf(h)(P)$. By \Cref{partition}, the set $h^{-1}(P)$ includes a prime filter $Q'$ containing $a$. Clearly $\D[Q'] = (h|_{\D[\algebra A]})^{-1}(\mu)$, and so $\D[Q'] =\D[Q]$. Then $Q' = Q$, by \Cref{intersection}. 
 Hence $Q \subseteq h^{-1}(P)$, that is, $Q \in \pf(h)(P)$. As $\mu$ was arbitrary, and $Q$ was arbitrary subject to $\D[Q] = (h|_{\D[\algebra A]})^{-1}(\mu)$, the relation $\pf(h)$ is star surjective.

For co-pseudo star surjectivity, suppose $U$ is an open set of arrows in $\pf(\algebra A)$. Since the topology on the arrows of $\pf(\algebra A)$ is generated by sets of the form $a^\theta$, we may assume $U = \bigcap_i a_i^\theta$ for some $a_1, \dots, a_n$. Now suppose $\mu$ is an ultrafilter of $\D[\algebra B]$ and $Q$ is a prime filter of $\algebra A$ with $Q \in U$ and $\R[Q]^\uparrow = (h|_{\D[\algebra A]})^{-1}(\mu)$. Then $a_1, \dots, a_n \in Q$, and as $Q$ is downward directed we can choose some $a \in Q$ with $a \leq a_1, \dots, a_n$. By the second hypothesis on $Q$, we know $\R(a) \in h^{-1}(\mu)$, so $h(\R(a)) \in \mu$. Since $h$ is a homomorphism, this gives $\R(h(a)) \in \mu$. By \Cref{filter-range}, there exists a prime filter $P$ containing $h(a)$ and such that $\R[P]^\uparrow = \mu$. Then $h^{-1}(P)$ contains $a$, so one of the prime filters that, by \Cref{partition}, partition $h^{-1}(P)$, contains $a$. Call this prime filter $Q'$. By upward closure, $a_1, \dots, a_n \in Q'$, hence $Q' \in U$. Since $Q'$ is in the image under $\pf(h)$ of $P \in \operatorname{Hom}(-, \mu)$, this completes the proof.
\end{proof}

\begin{lemma4}\label{functorial}
The map $\pf \from \mathscr A \to \mathscr C$ preserves composition of morphisms and identity morphisms and hence is a functor from $\mathscr A$ to $\mathscr C$.
\end{lemma4}

\begin{proof}
Let $h_1 \from \algebra A \to \algebra B$ and $h_2 \from \algebra B \to \algebra C$ be homomorphisms of representable $\{\compo, \A, \R, \pref\}$-algebras. Since the object components of $\pf(h_1)$, $\pf(h_2)$, and $\pf(h_2 \circ h_1)$ are given by inverse images of the induced maps $\D[\algebra A] \to \D[\algebra B] \to \D[\algebra C]$, the object components of $\pf(h_2) \circ \pf(h_1)$ and $\pf(h_2 \circ h_1)$ coincide. By \Cref{partition}, we can make the same conclusion for the arrow components. It is evident that $\pf$ applied to an identity map yields an identity map.
\end{proof}

The next lemma relates to the restricted duality of \Cref{restricted}.

\begin{lemma4}\label{proper0}
If $h$ is a locally proper homomorphism, the multivalued functor $\pf(h) \from \allowbreak\pf(\algebra B) \to \pf(\algebra A)$ is a functor.
\end{lemma4}

\begin{proof}
Conditions \ref{one}--\ref{three} of \Cref{multi-functor} reduce to the conditions defining a functor in the case that the relation component of the multivalued functor is (precisely) single valued. Hence we only need to know that $\pf(h)$ relates each prime filter $P$ of $\algebra{B}$ to precisely one prime filter of $\algebra{A}$. This is practically the definition of a locally proper homomorphism. For by the definition of a locally proper homomorphism, the set $h^{-1}(P)$ \emph{is} a prime filter. So by definition $\pf(h)$ relates $P$ to $h^{-1}(P)$. Conversely, if $\pf(h)$ relates $P$ to a prime filter $Q$, then by definition $Q \subseteq h^{-1}(P)$. Since $Q$ and $h^{-1}(P)$ are both prime filters, this gives $Q = h^{-1}(P)$.
\end{proof}

\section{From topological categories to algebras}\label{category_to_algebra}

In this section, we define the contravariant functor $\cpt \from \mathscr C \to \mathscr A$ used for the second half of the duality. The notation for the functor stands for `section on a clopen'.

\begin{definition}
Let $\pi \from E \to X$ be a local homeomorphism of topological spaces. A \defn{(local) section} of $\pi$ is a continuous function $f \from U \to E$, for some open $U \subseteq X$, with $\pi \circ f = \operatorname{id}_U$.
\end{definition}

Since a local section is completely determined by its image, we will often identify it with this image, in which case an upper-case Roman letter will be used.

\subsection{The functor \texorpdfstring{$\cpt$}{SecCl} on objects}

Let $\topo C$ be a Stone \'etale category with objects $O$ and arrows $M$ all of which are epimorphisms. Then we define $\cpt(\topo C)$ to be the following $\{\compo, \A, \R, \pref\}$-algebra.
\begin{itemize}
\item
 The universe of $\cpt(\topo C)$ is the set of all local sections $U \to M$ of $\dd \from M \to O$ with $U$ clopen.
 
\item
The operation $\compo$ is defined by $A \compo B \coloneqq \{a \cdot b \mid a \in A,\ b \in B,\text{ and }\rr(a) = \dd(b)\}$. The confirmation that $A \compo B$ is a section on a clopen is \Cref{clopen1}.

\item
The operation $\A$ is defined by $\A(A) \coloneqq \{1_x \mid x \in O \setminus \dd[A]\}$. The confirmation that $\A(A)$ is a section on a clopen is \Cref{clopen2}.

\item
The operation $\R$ is defined by $\R(A) \coloneqq \{1_x \mid x \in \rr[A]\}$. The confirmation that $\R(A)$ is a section on a clopen is \Cref{clopen3}.

\item
The operation $\pref$ is defined by $A \pref B \coloneqq A \cup (\A(A) \compo B)$. The confirmation that $A \pref B$ is a section on a clopen is \Cref{clopen4}.
\end{itemize}
Proving that $\cpt(\topo C)$ is representable by partial functions is achieved by verifying that it validates all the equations and quasiequations of \Cref{axiomatisation}; this is done in \Cref{validates}.



\begin{lemma5}\label{clopen2}
Let $A$ be a section with clopen domain. Then $\A(A) \coloneqq \{1_x \mid x \in O \setminus \dd[A]\}$ is a section with clopen domain.
\end{lemma5}

\begin{proof}
By definition, $\dd[A]$ is clopen in $O$, hence $O \setminus \dd[A]$ is clopen in $O$. So $\A(A)$ defines a function $f \from x \mapsto 1_x$ with clopen domain, clearly with left inverse $\dd$. It remains to show this function is continuous. But this is immediate, since $f$ is a restriction---given by a restriction of the domain---of the (continuous) identity-assigning map, and all such restrictions of continuous maps are continuous.
\end{proof}

\begin{lemma5}\label{open_image}
Let $\pi \from E \to X$ be a local homeomorphism and $U \subseteq X$ be open, and suppose $f \from U \to E$ has left inverse $\pi$. Then $f$ is continuous (that is, is a local section) if and only if $f[U]$ is open in $E$.
\end{lemma5}

\begin{proof}
For the forward direction let $e \in f[U]$. Then as $\pi$ is a local homeomorphism, there is some $V$ open in $E$ and containing $e$ such that $\pi|_V$ is a homeomorphism onto its image $\pi|_V[V] = \pi[V]$. As $f$ is continuous, the set $f^{-1}(V)$, which equals $f^{-1}(V \cap f[U])$, is open in $X$. But $f^{-1}(V \cap f[U]) \subseteq \pi[V]$, so as $\pi_V$ is a homeomorphism the inverse image of $f^{-1}(V \cap f[U])$ under $\pi|_V$, that is, $V \cap f[U]$, is open in $V$. Since $V$ itself is open, $V \cap f[U]$, which contains $e$, is open in $E$. Since $e$ was an arbitrary element of $f[U]$, the set $f[U]$ is open in $E$.

Conversely, suppose $f[U]$ is open. Let $V$ be an open subset of $E$, hence an open subset of $f[U]$. We want to show that $f^{-1}(V)$ is open. It suffices to show that any $x \in f^{-1}(V)$ is contained in an open neighbourhood included in $f^{-1}(V)$. But this is clear, since $\pi$ is a local homeomorphism, hence maps some open $W$ containing $f(x)$, and contained in $V$, to the open $\pi[W] = f^{-1}(W)$, which of course contains $x$ and is included in $V$.
\end{proof}

Note that \Cref{open_image} immediately implies that the identity arrows form an open set, since the identity-assigning map is manifestly a section.

\begin{lemma5}\label{clopen3}
Let $A$ be a section with clopen domain. Then $\R(A) \coloneqq \{1_x \mid x \in \rr[A]\}$ is a section with clopen domain.
\end{lemma5}

\begin{proof}
Let $A$ correspond to the function $f \from U \to M$. Then $\R(A)$ is the identity-assigning map restricted to $(\rr \circ f) [U]$. As $\R(A)$ is a restriction of a continuous function it is continuous. By \Cref{open_image} (with local homeomorphism set to $\dd$), the set $f[U]$ is open. Then since $\rr$ is an open map, $(\rr \circ f) [U]$ is open. It remains to show $(\rr \circ f) [U]$ is closed. Now $\rr \circ f$ is a continuous map from a compact space ($U$) to a Hausdorff space (the space of objects). It is a basic and easy-to-prove result of general topology that such a map is a closed map (the `closed map lemma'). Hence $(\rr \circ f) [U]$ is indeed closed.
\end{proof}

\begin{lemma5}\label{clopen1}
Let $A$ and $B$ be sections with clopen domains. Then \[A \compo B \coloneqq \{a \cdot b \mid a \in A,\ b \in B,\text{ and }\rr(a) = \dd(b)\}\] is a section with clopen domain.
\end{lemma5}

\begin{proof}
Let $A$ correspond to the function $f \from U \to M$ and $B$ to $g \from V \to M$. It is clear that $A \compo B$ corresponds to a function $h$ on a subset of $O$ and that $\dd$ is a left inverse for this function. Since $h$ can be expressed as a composition $g' \circ {\rr}' \circ f'$ of restrictions of the continuous functions $f$, $\rr$, and $g$, we see that $h$ is continuous. It remains to show that the domain of $h$ is clopen. By \Cref{clopen3}, the set $({\rr} \circ f)[U]$ is clopen; hence $({\rr} \circ f)[U] \cap V$ is clopen. Now the domain of $h$ is $({\rr} \circ f)^{-1}(({\rr} \circ f)[U] \cap V)$, so clopen by continuity of $f$ and $\rr$.
\end{proof}

\begin{lemma5}\label{clopen4}
If $A$ and $B$ are sections on clopens, then so is $A \pref B \coloneqq A \cup (\A(A) \compo B)$.
\end{lemma5}

\begin{proof}
It is immediate that $A \pref B$ defines a function with left inverse $\dd$. By \Cref{clopen1} and \Cref{clopen2}, we know $\A(A) \compo B$ is a section on a clopen, and hence the domain of $A \pref B$ is clopen. It remains to argue that $A \pref B$ is continuous. But this is a function given by the union of two continuous functions with open domains, which always yields a continuous function.
\end{proof}

\begin{lemma5}\label{validates}
The $\{\compo, \A, \R, \pref\}$-algebra $\cpt(\topo C)$ validates the axioms for representability by partial functions listed in \Cref{axiomatisation}.
\end{lemma5}

\begin{proof}
We state each axiom anew before giving a justification for its validity.
\[A \compo (B \compo C) = (A \compo B) \compo C\]
Clear.
\[\A(A) \compo A = \A(B) \compo B\]
Both sides yield the empty set.
\[A \compo \A(B) = \A(A \compo B) \compo A\]
Suppose $c \in A \compo \A(B)$. Then $c \in A$ and there is no arrow in $B$ whose source is $\rr(c)$. Since $c$ is the unique arrow in $A$ with $\dd(c)$, there is then no pair of composable arrows $a \in A$ and $b\in B$ with $\dd(a) = \dd(c)$. Hence $1_{\dd(c)}$ is in $ \A(A \compo B)$, so $c$ belongs to the right-hand side.

Conversely, suppose $c \in \A(A \compo B) \compo A$, then in particular there is no $(A, B)$-path from $\dd(c)$ so in particular, no arrow in $B$ whose source is $\rr(c)$. Hence $c \in A \compo \A(B)$. 
\[\D(A) \compo B = \D(A) \compo C \land \A(A) \compo B = \A(A) \compo C \qquad\limplies\qquad B = C\]
For any partition $I, J$ of the identity arrows into two parts, any set $D$ of arrows is the union of $I \compo D$ and $J \compo D$. Since $\D(A), \A(A)$ is such a partition, the validity of the quasiequation follows.
\[{\id} \compo A = A\]
We noted that $\A(A) \compo A$ yields the empty set. So $\id$ is by definition $\A(\emptyset)$---precisely the identity arrows. The equation is then clear.
\[0 \compo A = A\]
Clear.
\[\D(\R(A)) = \R(A)\]
The set $\R(A)$ is a set of identity arrows, and $\D \coloneqq \A^2$ is the identity operation on any such set.
\[A \compo \R(A) = A\]
By definition $\R(A)$ is all identities on objects that are the target of some arrow in $A$. So the validity of the equation is clear. 
\[A \compo B = A \compo C \qquad\limplies\qquad \R(A) \compo B = \R(A) \compo C\]
This is the only case of note, for we must use the fact that all arrows of $\topo C$ are epimorphisms. Assume the antecedent holds, and suppose $b \in \R(A) \compo B$. Then $b \in B$ and there is some $a\in A$ whose target is $\dd(b)$. So $a \cdot b \in A \compo B$ and hence, by the supposition, $a \cdot b \in A \compo C$. Hence $a \cdot b = a' \cdot c$, for some $a' \in A$ and $c \in C$, though necessarily $a = a'$, as $A$ is a section and $\dd(a) = \dd(a \ccompo b) = \dd(a' \ccompo c) = \dd(a')$. Since $a$ is an epimorphism, we obtain $b = c$, hence $b \in C$, and therefore $b \in \R(A) \compo C$. The reverse inclusion is by a symmetric argument.
\[\D(A) \compo (A \pref B) = A\]
If $a \in \D(A) \compo (A \pref B)$ then firstly there is an arrow in $A$ with the same source as $a$. Secondly, by the definition of $A \pref B$ on local sections, $a$ is in either $A$ or $\A(A) \compo B$. But $a$ cannot be in $\A(A) \compo B$, since that implies there \emph{is not} an arrow in $A$ with the same source as $a$. Hence $a \in A$. Conversely, if $a \in A$ then $a \in A \pref B$ and $1_{\dd(a)} \in \D(A)$, so $1_{\dd(a)} \cdot a = a \in \D(A) \compo (A \pref B)$.
\[\A(A) \compo (A \pref B) = \A(A) \compo B\]
If $a \in \A(A) \compo (A \pref B)$ then firstly there is no arrow in $A$ with the same source as $a$. Secondly, by the definition of $A \pref B$ on local sections, $a$ is in either $A$ or $\A(A) \compo B$---it must be $\A(A) \compo B$. Conversely, if $a \in \A(A) \compo B$ then immediately $a\in A \pref B$, using the definition of $\pref$ on local sections again. That $a \in \A(A) \compo B$ also implies $1_{\dd(a)} \in \A(A)$. Hence $a = 1_{\dd(a)} \cdot a  \in \A(A) \compo (A \pref B)$.
\end{proof}

\begin{example}
Returning again to our running example (\Cref{small} and \Cref{small_dual}), the reader may examine the (discrete) category $\pf(\algebra A)$ described in \Cref{small_dual} and calculate \emph{its} dual. Sections must contain at most one arrow with source `$1,2$' (there are three such arrows, so four possible choices) and at most one arrow with source `$3$' (one arrow, so two choices). In total there are eight sections, so eight element of the dual of $\pf(\algebra A)$. One can verify that this `double dual' is isomorphic to the original algebra $\algebra A$.
\end{example}

\subsection{The functor \texorpdfstring{$\cpt$}{SecCl} on morphisms} 

The action of $\cpt$ on morphisms is given by inverse image. That is, given $\topo C, \topo D$ Stone \'etale categories whose arrows are epimorphisms and a star-coherent multivalued functor $F \from \topo C \to \topo D$, we define 
\begin{align*}
\cpt (F) \from \cpt(\topo D) &\to \cpt(\topo C)\\
A &\mapsto F^{-1}(A).
\end{align*}
First it must be checked that $F^{-1}(A)$ really is an element of $\cpt (\topo C)$. If $b_1, b_2 \in F^{-1}(A)$ have the same source, then as $F$ is a multivalued functor, arrows in $F(b_1)$ and $F(b_2)$ all share the same source. As $A$ is a section, $b_1$ and $b_2$ must be inverse images under $F$ of the same $a \in A$. Then by star injectivity of $F$, we get $b_1 = b_2$. Hence $F^{-1}(A)$ defines a function on $\dd[F^{-1}(A)]$. As $F$ is by assumption continuous, and $\dd[A]$ is clopen, $F^{-1}(\dd[A]) = \dd[F^{-1}(A)]$ (star surjectivity) is clopen. Since $F^{-1}(A)$ is open, \Cref{open_image} implies that it corresponds to a \emph{continuous} function. That $\cpt(F)$ is a homomorphism of $\{\compo, \A, \R, \pref\}$-algebras is \Cref{homomorphism}. It is immediate from its definition by inverse image that $\cpt$ is functorial, that is, respects compositions of star-coherent multivalued functors and acts as the identity on identity functors.

\begin{lemma5}\label{homomorphism}
The map $\cpt (F) \from \cpt(\topo D) \to \cpt(\topo C)$ of $\{\compo, \A, \R, \pref\}$-algebras is a homomorphism.
\end{lemma5}

\begin{proof}
We start by showing that $\compo$ is preserved. Let $A$ and $B$ be sections of $\topo D$. If $c \in F^{-1}(A) \compo F^{-1}(B)$ that means $c = c_1 \ccompo c_2$ for some $c_1, c_2$ such that there exist $a \in A \cap F(c_1)$ and $b \in B \cap F(c_2)$. Then by functoriality of $F$, we know $a \ccompo b \in F(c)$, so $a \ccompo b \in (A \compo B) \cap F(c)$. That is, $c \in F^{-1}(A \compo B)$. Hence $F^{-1}(A) \compo F^{-1}(B) \subseteq F^{-1}(A \compo B)$. Conversely, if $c \in F^{-1}(A \compo B)$, with $ a \cdot b \in F(c)$ say, then by star surjectivity applied at $\dd (a)$ there is some $c_1$ with the same source as $c$ and with $a \in F(c_1)$. Applying star surjectivity again at $\rr(a) = \dd(b)$ we obtain some $c_2$ with $c_1 \cdot c_2 = c$ and $b \in F(c_2)$. Hence $c \in F^{-1}(A) \compo F^{-1}(B)$, and we conclude that $F^{-1}(A) \compo F^{-1}(B) = F^{-1}(A \compo B)$.

Next we show that $\A$ is preserved. Let $B$ be a section of $\topo D$. First suppose $a \in F^{-1}(\A(B))$, with $1_y \in F(a) \cap \A(B)$ say. Then by functoriality of $F$, the identity arrow $1_y$ belongs to $F(1_{\dd(a)})$. By star injectivity of $F$, we find $a = 1_{\dd(a)}$. Since $1_y \in \A(B)$ there is no member of $B$ with source $y$, hence there is no member of $F^{-1}(B)$ with source $\dd(a)$. So by definition, $1_{\dd(a)} \in \A(F^{-1}(B))$, that is, $a \in \A(F^{-1}(B))$. We have our first inclusion: $F^{-1}(\A(B)) \subseteq \A(F^{-1}(B))$. Conversely, suppose $1_x \in \A(F^{-1}(B))$. We want to argue that $1_{F(x)}$ is in $\A(B)$. But if there were a member $b$ of $B$ with source $F(x)$, then by star surjectivity of $F$, there would be an arrow with source $x$ whose image under $F$ contains $b$, contradicting the fact that $1_x \in \A(F^{-1}(B))$. Hence $1_{F(x)}$ is indeed in $\A(B)$, and since $1_{F(x)} \in F(1_x)$ this gives $1_x \in F^{-1}(\A(B))$. We conclude that $F^{-1}(\A(B)) = \A(F^{-1}(B))$.

Showing that $\R$ is preserved is fairly similar. Let $B$ be a section of $\topo D$ with a clopen domain. By \Cref{open_image}, the set $B$ is open. First suppose we have an element in $F^{-1}(\R(B))$. By functoriality and star injectivity of $F$, our element of $F^{-1}(\R(B))$ is an identity element, $1_y$ say. By definition of $\R$, there is some $b \in B$ with target $F(y)$. As $B$ is open, by co-pseudo star surjectivity of $F$ there is some $b' \in B$ belonging to $F(a)$ for some $a$ with target $y$. So $a \in F^{-1}(B)$, hence $1_y \in \R(F^{-1}(B))$. We conclude that $F^{-1}(\R(B)) \subseteq \R(F^{-1}(B))$. 
Conversely, suppose $1_y \in \R(F^{-1}(B))$. So there is an $a$ with target $y$ such that $F(a) \cap B \neq \emptyset$. Then $1_{F(y)} \in \R(B)$, so $1_y \in F^{-1}(\R(B))$. We conclude that $F^{-1}(\R(B)) = \R(F^{-1}(B))$.

Finally we note that $\pref$ is preserved, because of the definition $A \pref B \coloneqq A \cup \A(A) \compo B$ and the elementary fact that relation inverse images preserve unions.
\end{proof}

\section{The functors form a duality}\label{are_dual}

In this section we will first show that the double dual functor on the category $\mathscr A$ of	representable $\{\compo, \A, \R, \pref\}$-algebras is naturally isomorphic to the identity functor. Then we will do the same for the double dual functor on the category $\mathscr C$ of Stone \'etale categories all of whose arrows are epimorphisms. This will complete the proof that we have given a duality between the categories $\mathscr A$ and $\mathscr C$.

\subsection{The double dual on algebras}

First we describe an isomorphism from $\algebra A$ to $\cpt(\pf(\algebra A))$ for an arbitrary representable $\{\compo, \A, \R, \pref\}$-algebra. Then we will show this construction is natural.

In fact, our isomorphism is already hidden in notation we have defined. Recall that for $a \in \algebra A$, the set $a^\theta$ is defined to be $\{P \in \operatorname{pf}(\algebra A) \mid a \in P\}$. We define $\theta \from \algebra A \to \cpt(\pf(\algebra A))$ by $a \mapsto a^\theta$. Note that $a^\theta$ is indeed an element of the algebra $\cpt(\pf(\algebra A))$, for it is clearly a section, and its domain is the set $\widehat {\D(a)} \coloneqq \{\mu \in \operatorname{uf}(\D[\algebra A]) \mid {\D(a)} \in \mu\}$, which is open by definition and closed because $\widehat {\A(a)}$ is open.

To see that $\theta$ is injective, it suffices to show that when $a \not\leq b$ there exists a prime filter containing $a$ but not $b$. This argument can be found in the proof of Lemma~4.9 in \cite{hirsch}. To see that $\theta$ is surjective, we need to argue that \emph{all} sections on clopens of $\pf(\algebra A)$ are of the form $a^\theta$. Let $A$ be a section on a clopen of $\pf(\algebra A)$. By a similar argument to that in the proof of \Cref{continuous_functor}, for each $P \in A$ there is an $a_P$ with $P \in a_P^\theta \subseteq A$. Recall that the space of objects of $\pf(\algebra A)$ is the Stone dual of $\D[\algebra A]$, so in particular is compact. Hence the domain of $A$ is compact and can be covered by $\widehat{\D(a_{P_1})}, \dots, \widehat{\D(a_{P_n})}$ for some finite $n$. Then $A = (a_{P_1} \pref \dots \pref a_{P_n})^\theta$.

\begin{lemma}
The map $\theta \from \algebra A \to \cpt(\pf(\algebra A))$ given by $a \mapsto a^\theta$ is a homomorphism of $\{\compo, \A, \R, \pref\}$-algebras.
\end{lemma}

\begin{proof}
To show $\compo$ is preserved, let $a, b \in \algebra A$. If $P \in a^\theta \compo b^\theta$, that means there are $P_1$ containing $a$ and $P_2$ containing $b$ such that $(P_1 \compo P_2)^\uparrow = P$. Hence $a \compo b \in P$, giving $P \in (a \compo b)^\theta$. Conversely, suppose $P \in (a \compo b)^\theta$. Then $P$ is of the form $(\mu \compo a \compo b)^\uparrow$, where $\mu$ is the ultrafilter $D[P]$ of $\D[\algebra A]$. As $P$ does not contain $0$, neither does $(\mu \compo a)^\uparrow$, which is therefore a prime filter, $P_1$ say. Let $\nu$ be the ultrafilter $\R[P_1]^\uparrow$. Using again the fact that $P = (\mu \compo a \compo b)^\uparrow = (\mu \compo a \compo \nu \compo b)^\uparrow$ does not contain $0$, the filter $(\nu \compo b)^\uparrow$ is a prime filter, $P_2$ say. Then $P_1 \in a^\theta$ and $P_2 \in b^\theta$ with $P_1 \ccompo P_2 = P$, hence $P \in a^\theta \compo b^\theta$.

To show $\A$ is preserved, let $a \in \algebra A$. If $P \in \A(a^\theta)$, then $P = \mu^\uparrow$ for some ultrafilter $\mu$ of $\D[\algebra A]$, and there is no prime filter containing $a$ with source $\mu$. Hence the filter $(\mu \compo a)^\uparrow$ is not proper. That is, there is some $\alpha \in \mu$ such that $\alpha \compo a = 0$. It is a property of partial functions that this implies $\alpha \leq \A(a)$, so, by upward closure of $\mu^\uparrow$, we know $\A(a) \in \mu^\uparrow = P$. Hence $P \in \A(a)^\theta$. Conversely, suppose $P \in \A(a)^\theta$. Then $\A(a) \in P$ so ${\id} \in P$. Hence $P$ is of the form $(\mu \compo \id)^\uparrow = \mu^\uparrow$ for some ultrafilter $\mu$. Hence $P$ is the identity arrow for $\mu$ in the category $\pf(\algebra A)$. We know $\A(a) \in \mu$. If there were a prime filter $P'$ containing $a$ with source $\mu$ then $\D(a)$ would also be in $\mu$, contradicting $\mu$ being proper. Hence there is no such $P'$, and we conclude $P \in \A(a^\theta)$.

To show $\R$ is preserved, let $a \in \algebra A$. If $P \in \R(a^\theta)$, then $P = \mu^\uparrow$ for some ultrafilter $\mu$ of $\D[\algebra A]$, and there is some prime filter $P'$ containing $a$ with target $\R[P]^\uparrow = \mu$. Then $\R(a) \in \mu$, so $\R(a) \in \mu^\uparrow = P$. Hence $P \in \R(a)^\theta$. Conversely, suppose $P \in \R(a)^\theta$. Then $\R(a) \in P$ so ${\id} \in P$. Hence $P$ is of the form $(\mu \compo \id)^\uparrow = \mu^\uparrow$ for some ultrafilter $\mu$. Hence $P$ is the identity arrow for $\mu$ in the category $\pf(\algebra A)$. We know $\R(a) = \D(\R(a)) \in \D[P] = \mu$. By \Cref{filter-range} there exists a prime filter $P$ containing $a$ with target $\mu$. We conclude $P \in \R(a^\theta)$.

To see that $\pref$ is preserved, note that if a prime filter contains $a \pref b = a \pref (\A(a) \compo b)$ it contains $a$ or $\A(a) \compo b$ (by primality), and if it contains $a$ or $\A(a) \compo b$ then it contains $a \pref b$ (by upward closure). Thus $(a \pref b)^\theta = a^\theta \cup (\A(a) \compo b)^\theta$. The latter, given we know $\compo$ and $\A$ are preserved by $\theta$, equals $a^\theta \cup (\A(a^\theta) \compo b^\theta)$, which by definition is $a^\theta \pref b^\theta$.
\end{proof}

We now show that our isomorphisms together give a natural transformation from the identity functor to the double dual. For each $\algebra A \in \mathscr A$, denote now the isomorphism just described by $\theta_{\algebra A}$. Then given $\algebra A, \algebra B \in \mathscr A$ and a homomorphism $h \from \algebra A \to \algebra B$, we are required to show that $\cpt(\pf(h)) \circ \theta_{\algebra A} = \theta_{\algebra B} \circ h$. The right-hand side sends an element $a \in \algebra A$ to the set $h(a)^\theta$ of prime filters of $\algebra B$. Seeing that the left-hand side has the same effect just involves unravelling the definitions. The element $a$ is sent first to $a^\theta$, then $\cpt(\pf(h))$ sends this to
\begin{align*}
\{P \in \pf(\algebra B) \mid \pf(h)(P) \in a^\theta\} &= \{P \in \pf(\algebra B) \mid h^{-1}(P) \in a^\theta\}\\
&=\{P \in \pf(\algebra B) \mid a \in h^{-1}(P)\}\\
&=\{P \in \pf(\algebra B) \mid h(a) \in P\}\\
&= h(a)^\theta
\end{align*}
as required.

\subsection{The double dual on categories}

Let $\topo C \in \mathscr C$, and write $\operatorname{Id}(\cpt(\topo C))$ for the elements of $\cpt(\topo C)$ consisting entirely of identity arrows.	Define $\varphi \from \topo C \to \pf(\cpt(\topo C))$ by:
\begin{itemize}
\item
for an object $x$:
\[x \mapsto x^\varphi \coloneqq \{A \in \operatorname{Id}(\cpt(\topo C)) \mid 1_x \in A\},\]

\item
for an arrow $c$:
\[c \mapsto c^\varphi \coloneqq \{A \in \cpt(\topo C) \mid c \in A\}.\]

\end{itemize}
We first verify that $x^\varphi$ is an ultrafilter of $\D[\cpt(\topo C)]$, and $c^\varphi$ is a prime filter of $\cpt(\topo C)$, so $\varphi$ indeed has codomain $\pf(\cpt(\topo C))$.

 For $x^\varphi$, the identity-assigning map is a section on a clopen, so $x^\varphi$ is nonempty. If $A, B \in \operatorname{Id}(\cpt(\topo C))$ then $A \cap B \in \operatorname{Id}(\cpt(\topo C))$, so $x^\varphi$ is downward directed. It is trivial that $x^\varphi$ is upward closed and clear that for $A \in \operatorname{Id}(\cpt(\topo C))$ precisely one of $A$ and $\A(A)$ is in $x^\varphi$.

 For $c^\varphi$, as $\dd$ is a local homeomorphism from the arrows to the objects of $\topo C$, there exists some section $s$ on an open  that has $c$ in its image. Since the space of objects of $\topo C$ has a basis of clopens, we can restrict $s$ to a clopen still with $c$ in its image. Hence $c^\varphi$ is nonempty. By a similar argument, $c^\varphi$ is down directed. It is trivial that $c^\varphi$ is upward closed and straightforward that it satisfies the primality condition. Clearly $\emptyset \not\in c^\varphi$, and $\emptyset$ is the $0$ of $\cpt(\topo C)$, so $c^\varphi$ is proper.

\begin{lemma}
The map $\varphi$ is a functor between the categories $\topo C$ and $\pf(\cpt(\topo C))$.
\end{lemma}

\begin{proof}
Let $(c \from x \to y) \in \topo C$. To see that $c^\varphi \from x^\varphi \to y^\varphi$, suppose $A \in c^\varphi$. Then $c \in A$, so $x = \dd(c) \in \dd[A]$, so $1_x \in \D(A)$. Hence $\D(A) \in x^\varphi$. We conclude that $\D[c^\varphi] \subseteq x^\varphi$. Since both are ultrafilters, they are equal. Hence $\dd(c^\varphi) = \D[c^\varphi] = x^\varphi$. Similarly, we have $y = \rr(c) \in \rr[A]$, so $1_y \in \R(A)$. Hence $\R(A) \in y^\varphi$. We conclude that $\R[c^\varphi] \subseteq y^\varphi$. Since the ultrafilter $y^\varphi$ is upward closed, we get $\R[c^\varphi]^\uparrow \subseteq y^\varphi$, so since both these are ultrafilters, they are equal. Hence $\rr(c^\varphi) = \R[c^\varphi]^\uparrow = y^\varphi$.

Next we argue that for any object $x$, we have $(1_x)^\varphi = 1_{x^\varphi}$. The right-hand side is by definition $(x_\varphi)^\uparrow$. That is, an element $A$ of $1_{x^\varphi}$ is an upper bound for (that is, superset of) some identity section $A'$ that contains $1_x$. Hence $A$ itself contains $1_x$, so $A \in (1_x)^\varphi$. We conclude $(1_x)^\varphi \supseteq 1_{x^\varphi}$. Since both are prime filters, they are equal.

Let $(c_1 \from x \to y), (c_2 \from y \to z) \in \topo C$. We know that $c_1^\varphi \ccompo c_2^\varphi \from x \to z$ is defined, in particular is a prime filter. To see that $c_1^\varphi \ccompo c_2^\varphi = (c_1 \ccompo c_2)^\varphi$, we have 
\begin{align*}
&A \in c_1^\varphi\text{ and }B \in c_2^\varphi\\
\implies & c_1 \in A\text{ and }c_2 \in B\\
\implies & c_1 \ccompo c_2 \in A \compo B\\
\implies & A \compo B \in (c_1 \ccompo c_2)^\varphi.
\end{align*}
Hence $c_1^\varphi \ccompo c_2^\varphi \subseteq (c_1 \ccompo c_2)^\varphi$. Since both are prime filters, they are equal.
\end{proof}

To see that $\varphi$ is injective (on arrows, therefore on objects) let $c \neq d \in \topo C$. Choose any section on a clopen $A$ that contains $c$. If $\dd(c) = \dd(d)$, then $A$ cannot contain $d$, so $c^\varphi$ and $d^\varphi$ are not equal. If $\dd(c) \neq \dd(d)$, then we can find a clopen set $U$ of objects that contains $\dd(c)$ but not $\dd(d)$. By restricting $A$ to $U$, we obtain a section on a clopen containing $c$ but not $d$, so again $c^\varphi$ and $d^\varphi$ are not equal. 

To see that $\varphi$ is surjective (on arrows, therefore on objects), take a prime filter $P$ of sections on clopens of $\topo C$. Let $S = \bigcap P$. If $c \in S$ then $P \subseteq c^\varphi$, and so $P = c^\varphi$. Hence we only need to show $S$ cannot be empty. Let $A \in P$. As $P$ is a filter, $\bigcap P = \emptyset$ implies $\bigcap \{ B \in P \mid B \subseteq A\} = \emptyset$. As $A$ is compact (being the continuous image of a compact set), this implies $B_1 \cap \dots \cap B_n = \emptyset$ for some $B_1, \dots, B_n \in \{ B \in P \mid B \subseteq A\}$. As $P$ is downward directed, this implies $\emptyset \in P$---the required contradiction.

Since $\varphi$ is bijective, it is certainly star-coherent. To show $\varphi$ is an isomorphism in $\mathscr C$ it remains to show that $\varphi$ and its inverse are continuous. First we need a lemma.

\begin{lemma6}\label{subbasis}
In any Stone \'etale category, the (images of) sections on clopens provide a basis for the topology on the arrows.
\end{lemma6}

\begin{proof}
Let $U$ be an open set of arrows, and suppose $c \in U$. As $\dd$ is a local homeomorphism, $c$ has an open neighbourhood $V$, which we may assume is a subset of $U$, such that ${\dd}|_{V}$ provides a homeomorphism onto its image, which is also open. That is, ${\dd}|_V^{-1}$ is a section on an open. As the set of objects has a clopen basis, we may restrict ${\dd}|_V^{-1}$ to a clopen containing $\dd(c)$, giving the section on a clopen containing $c$ and included in $U$ that we seek.
\end{proof}

By the lemma, to show that $\varphi$ is an open map, it suffices to consider an arbitrary section on a clopen $A$ of $\topo C$. We claim that $\varphi[A]$ equals $A^\theta$ and is therefore in particular open. If $c^\varphi \in \varphi[A]$ (for $c \in A$), then as $A$ is a section on a clopen, $A \in c^\varphi$, so $c^\varphi \in A^\theta$. We conclude that $\varphi[A] \subseteq A^\theta$. Now let $P \in A^\theta$, and suppose for a contradiction that $P \not\in \varphi[A]$. Since for prime filters (which are maximal filters) inclusion implies equality, $P \not\in \varphi[A]$ implies for all $c \in A$ we have $c^\varphi \not\subseteq P$. That is, we can find, for each $c \in A$, a section on a clopen $B_c$ that contains $c$, but is not in $P$. We may assume (by the same reasoning as in the proof of \Cref{subbasis}) that each $B_c$ is a subset of $A$. Then as $A$ is compact, some finite collection $B_{c_1}, \dots, B_{c_n}$ cover $A$. That is, $A = B_{c_1} \pref \dots \pref B_{c_n}$ in $\cpt(\topo C)$. As $P$ is prime and contains $A$, it must contain some $B_{c_i}$---the required contradiction.

Continuity of $\varphi$ now follows straightforwardly. That $\varphi[A] = A^\theta$ and $\varphi$ is injective implies $\varphi^{-1}(A^\theta) = A$, which is open if $A$ is a section on a clopen. The set of $A^\theta$'s such that $A$ is a section on a clopen provides a basis for $\pf(\cpt(\topo C))$, so we are done.

We now show that our isomorphisms together give a natural transformation from the identity functor to the double dual. For each $\topo C \in \mathscr C$, denote now the isomorphism just described by $\varphi_{\topo C}$. Then given $\topo C, \topo D \in \mathscr C$ and a star-coherent multivalued functor $F \from \topo C \to \topo D$, we are required to show that $\pf(\cpt(F)) \circ \varphi_{\topo C} = \varphi_{\topo D} \circ F$ (as multivalued functors). The right-hand side sends an arrow $c \in \topo C$ to $\{d^\varphi \mid d \in F(c)\}$. On the left-hand side, $c$ is first sent to $c^\varphi$, then $\pf(\cpt(F))$ sends this to the set of prime filters of $\cpt(F)$ that partition
\begin{align*}
\{A \in \cpt(\topo D) \mid \cpt(F)(A) \in c^\varphi\} &= \{A \in \cpt(\topo D) \mid F^{-1}(A) \in c^\varphi\}\\
&=\{A \in \cpt(\topo D) \mid c \in F^{-1}(A)\}\\
&=\{A \in \cpt(\topo D) \mid F(c) \cap A \neq \emptyset\}.
\end{align*}
And as the prime filters of $\cpt(F)$ are of the form $d^\varphi$,  the set of prime filters partitioning $\{A \in \cpt(\topo D) \mid F(c) \cap A \neq \emptyset\}$ is $\{d^\varphi \mid d \in F(c)\}$, exactly as required. This completes the proof of \Cref{main-theorem}.

We now complete the proof of the restricted duality of \Cref{restricted}.

\begin{lemma5}\label{proper}
If $F$ is a functor, the homomorphism $\cpt (F) \from \cpt(\topo D) \to \cpt(\topo C)$ of $\{\compo, \A, \R, \pref\}$-algebras is locally proper.
\end{lemma5}

\begin{proof}
Take $P$ to be an arbitrary prime filter in $\cpt(\topo C)$. By the preceding discussion we know $P$ is of the form $c^\varphi$ for some arrow $c$ of $\topo C$. So 
\begin{align*}\cpt(F)^{-1}(P) &= \{A \in \cpt(\topo D) \mid c \in \cpt(F)(A)\}\\
& = \{A \in \cpt(\topo D) \mid c \in F^{-1}(A)\}\\
& = \{A \in \cpt(\topo D) \mid F(c) \in A\}\\
& = F(c)^\varphi,
\end{align*}
which we know is a prime filter of $\cpt(\topo D)$.
\end{proof}

Now to finish the proof of \Cref{restricted}, note that by \Cref{proper0} and \Cref{proper}, the double dual of any \emph{locally proper} homomorphism is \emph{locally proper}, and the double dual of any star-coherent \emph{functor} is a  \emph{functor}. Then $\theta$ and $\varphi$ provide the required natural isomorphisms, since isomorphisms of representable $\{\compo, \A, \R\}$-algebras are locally proper, and isomorphisms of topological categories are functors.

\section{Word-to-word functions}\label{applications}

Let $\Sigma$ be a finite alphabet. Recall that the \defn{rational} functions over $\Sigma$, which we denote $\operatorname{Rat_f}(\Sigma)$, are the partial functions from $\Sigma^*$ to $\Sigma^*$ realisable by a one-way transducer. The \defn{regular} functions over $\Sigma$, which we denote $\operatorname{Reg_f}(\Sigma)$, are the partial functions from $\Sigma^*$ to $\Sigma^*$ realisable by a two-way transducer. (See \cite{10.1145/2984450.2984453} for an overview of these concepts.) The rational and the regular functions are both closed under $\compo$, $\A$, $\R$, and $\pref$ and hence are both $\{\compo, \A, \R, \pref\}$-algebras of partial functions with base $\Sigma^*$. Clearly
\[\operatorname{Rat_f}(\Sigma) \subseteq \operatorname{Reg_f}(\Sigma),\]
and so $\operatorname{Rat_f}(\Sigma)$ is a subalgebra of $\operatorname{Reg_f}(\Sigma)$. For both $\operatorname{Rat_f}(\Sigma)$ and $\operatorname{Reg_f}(\Sigma)$, the subalgebra of subidentity functions is the set of identity functions on \emph{regular languages} (by a simple `forgetting the output' argument). The Stone dual of the regular languages over $\Sigma$ is known to be (the underlying space of) the profinite completion $\widehat{\Sigma^*}$ of the monoid $\Sigma^*$ \cite{Pippenger_1997}. Hence both $\operatorname{Rat_f}(\Sigma)$ and $\operatorname{Reg_f}(\Sigma)$ have duals whose space of objects is $\widehat{\Sigma^*}$.

In the setting of \emph{languages}, duality provides a powerful and, in principle, fully general method for characterising any family of languages that forms a sublattice of the regular languages \cite{gehrke2008duality, gehrke2009stone}. For example, for the subalgebra of star-free languages this yields the characterisation by the profinite equation $x^\omega x = x^\omega$,  equivalent to Sch\"utzenberger's celebrated characterisation by aperiodicity of the syntactic monoid of the given language \cite{schutzenberger1965finite}. It would be useful to have similar tools available for regular functions. 

Since the category of representable $\{\compo, \A, \R, \pref\}$-algebras with homomorphisms is a concrete category, the embedding $\operatorname{Rat_f}(\Sigma) \hookrightarrow \operatorname{Reg_f}(\Sigma)$ is a monomorphism in that category. 

\begin{problem}
Is the embedding $\operatorname{Rat_f}(\Sigma) \hookrightarrow \operatorname{Reg_f}(\Sigma)$ locally proper?
\end{problem}

The answer is no, by the following general result.

\begin{proposition}\label{isomorphism}
Let $h \from \algebra A \to \algebra B$ be a locally proper homomorphism of $\{\compo, \A, \R, \pref\}$-algebras such that the induced map $\D[\algebra A] \to \D[\algebra B]$ is an isomorphism. Then $h$ is an isomorphism.
\end{proposition} 

\begin{proof}
If the hypotheses hold, then the dual $F \from \pf(\algebra B) \to \pf(\algebra A)$ of $h$ is a functor and a bijection on objects. But any star-coherent functor that is bijective on objects must be bijective on arrows, and thus $F$ is an algebraic isomorphism of categories. It follows in particular (surjectivity of $F$) that $h$ is injective, so we may assume $h$ is an inclusion of a subalgebra, and the induced $\D[\algebra A] \to \D[\algebra B]$ is the identity.

We know that $F$ is continuous, so to show that $F$ is an isomorphism of \emph{topological} categories it only remains to show $F$ is an open map. From there the conclusion that $h$ is an isomorphism is immediate, by duality.

The topology on $\pf(\algebra B)$ is generated by sets of the form $\theta_{\algebra B}(b)$ for $b \in \algebra B$. We must show that each $F[\theta_{\algebra B}(b)]$ is open in $\pf(\algebra A)$. For $P \in \algebra B$, we have $F(P) = h^{-1}(P) = P \cap \algebra A$. Let $\mu = \D[P] = \D[F(P)]$. Choose some $a \in F[P] \subseteq P$. We know $F[P] = (\mu \compo a)^\uparrow$, with upward closure taken in $\algebra A$, and $P =(\mu \compo a)^\uparrow$, with upward closure taken in $\algebra B$. That is the algebraic inverse $F^{-1}$ of $F$ is given by taking the upward closure in $\algebra B$. So, if $b \in P$ then there is some $a \leq b \in P \cap \algebra A$, that is, $F(P) \in \theta_{\algebra A}(a)$, and $\theta_{\algebra A}(a)$ is included in $F[\theta_{\algebra B}(b)]$. Hence $F[\theta_{\algebra B}(b)]$ is open, for arbitrary $b \in \algebra B$.
\end{proof}

Since $\operatorname{Rat_f}(\Sigma)$ and $\operatorname{Reg_f}(\Sigma)$ have the same Boolean subalgebra of domain ele\-ments---the regular languages encoded as subidentity functions---if the embedding $\operatorname{Rat_f}(\Sigma) \hookrightarrow \operatorname{Reg_f}(\Sigma)$ were locally proper, \Cref{isomorphism} would apply. But $\operatorname{Rat_f}(\Sigma)$ is strictly included in $\operatorname{Reg_f}(\Sigma)$, hence the embedding cannot be locally proper.

We conclude with a problem.

\begin{problem}
Give descriptions of the duals of $\operatorname{Rat_f}(\Sigma)$ and $\operatorname{Reg_f}(\Sigma)$ and of the dual of the embedding $\operatorname{Rat_f}(\Sigma) \hookrightarrow \operatorname{Reg_f}(\Sigma)$.
\end{problem}

\bibliography{../brettbib}{}

\begin{thebibliography}{10}

\bibitem{ABRAMSKY19911}
Samson Abramsky.
\newblock Domain theory in logical form.
\newblock {\em Annals of Pure and Applied Logic}, 51(1):1--77, 1991.
\newblock \href {https://doi.org/10.1016/0168-0072(91)90065-T}
  {\path{doi:10.1016/0168-0072(91)90065-T}}.

\bibitem{Bauer_2013}
Andrej Bauer, Karin Cvetko-Vah, Mai Gehrke, Samuel~J. van Gool, and Ganna
  Kudryavtseva.
\newblock A non-commutative {P}riestley duality.
\newblock {\em Topology and its Applications}, 160(12):1423--1438, 2013.
\newblock \href {https://doi.org/10.1016/j.topol.2013.05.012}
  {\path{doi:10.1016/j.topol.2013.05.012}}.

\bibitem{blackburn_rijke_venema_2001}
Patrick Blackburn, Maarten~de Rijke, and Yde Venema.
\newblock {\em Modal Logic}.
\newblock Cambridge Tracts in Theoretical Computer Science. Cambridge
  University Press, 2001.

\bibitem{DOCHERTY2018101}
Simon Docherty and David Pym.
\newblock A {S}tone-type duality theorem for separation logic via its
  underlying bunched logics.
\newblock {\em Electronic Notes in Theoretical Computer Science}, 336:101--118,
  2018.
\newblock The 33rd Conference on the Mathematical Foundations of Programming
  Semantics.
\newblock \href {https://doi.org/10.1016/j.entcs.2018.03.018}
  {\path{doi:10.1016/j.entcs.2018.03.018}}.

\bibitem{1018.20057}
Wies{\l}aw~A. Dudek and Valentin~S. Trokhimenko.
\newblock {Functional Menger $\mathcal P$-algebras}.
\newblock {\em Communications in Algebra}, 30(12):5921--5931, 2002.
\newblock \href {https://doi.org/10.1081/AGB-120016022}
  {\path{doi:10.1081/AGB-120016022}}.

\bibitem{10.2307/27588391}
J.~Michael Dunn, Mai Gehrke, and Alessandra Palmigiano.
\newblock Canonical extensions and relational completeness of some
  substructural logics.
\newblock {\em The Journal of Symbolic Logic}, 70(3):713--740, 2005.
\newblock \href {https://doi.org/10.2178/jsl/1122038911}
  {\path{doi:10.2178/jsl/1122038911}}.

\bibitem{esakia}
Leo~L. Esakia.
\newblock Topological {K}ripke models.
\newblock {\em Doklady Akademii Nauk}, 214(2):298--301, 1974.

\bibitem{10.1145/2984450.2984453}
Emmanuel Filiot and Pierre-Alain Reynier.
\newblock Transducers, logic and algebra for functions of finite words.
\newblock {\em ACM SIGLOG News}, 3(3):4--19, August 2016.
\newblock \href {https://doi.org/10.1145/2984450.2984453}
  {\path{doi:10.1145/2984450.2984453}}.

\bibitem{gehrke2009stone}
Mai Gehrke.
\newblock Stone duality and the recognisable languages over an algebra.
\newblock In {\em International Conference on Algebra and Coalgebra in Computer
  Science}, pages 236--250. Springer, 2009.
\newblock \href {https://doi.org/10.1007/978-3-642-03741-2_17}
  {\path{doi:10.1007/978-3-642-03741-2_17}}.

\bibitem{gehrke2008duality}
Mai Gehrke, Serge Grigorieff, and Jean-{\'E}ric Pin.
\newblock Duality and equational theory of regular languages.
\newblock In {\em International Colloquium on Automata, Languages, and
  Programming}, pages 246--257. Springer, 2008.
\newblock \href {https://doi.org/10.1007/978-3-540-70583-3_21}
  {\path{doi:10.1007/978-3-540-70583-3_21}}.

\bibitem{10.2307/24493402}
Mai Gehrke and Bjarni J\'onsson.
\newblock Bounded distributive lattice expansions.
\newblock {\em Mathematica Scandinavica}, 94(1):13--45, 2004.
\newblock \href {https://doi.org/10.7146/math.scand.a-14428}
  {\path{doi:10.7146/math.scand.a-14428}}.

\bibitem{GEHRKE2014290}
Mai Gehrke, Samuel~J. van Gool, and Vincenzo Marra.
\newblock Sheaf representations of {MV}-algebras and lattice-ordered abelian
  groups via duality.
\newblock {\em Journal of Algebra}, 417:290--332, 2014.
\newblock \href {https://doi.org/10.1016/j.jalgebra.2014.06.031}
  {\path{doi:10.1016/j.jalgebra.2014.06.031}}.

\bibitem{goldblatt}
Robert Goldblatt.
\newblock {\em Metamathematics of modal logic}.
\newblock PhD thesis, Victoria University of Wellington, April 1974.

\bibitem{hirsch}
Robin Hirsch, Marcel Jackson, and Szabolcs Mikul{\'a}s.
\newblock The algebra of functions with antidomain and range.
\newblock {\em Journal of Pure and Applied Algebra}, 220(6):2214--2239, 2016.
\newblock \href {https://doi.org/10.1016/j.jpaa.2015.11.003}
  {\path{doi:10.1016/j.jpaa.2015.11.003}}.

\bibitem{DBLP:journals/ijac/JacksonS11}
Mar{c}el Jackson and Tim Stokes.
\newblock Modal restriction semigroups: towards an algebra of functions.
\newblock {\em International Journal of Algebra and Computation},
  21(7):1053--1095, 2011.
\newblock \href {https://doi.org/10.1142/S0218196711006844}
  {\path{doi:10.1142/S0218196711006844}}.

\bibitem{Lawson_2010}
Mark~V. Lawson.
\newblock A noncommutative generalization of {S}tone duality.
\newblock {\em Journal of the Australian Mathematical Society}, 88(3):385--404,
  2010.
\newblock \href {https://doi.org/10.1017/s1446788710000145}
  {\path{doi:10.1017/s1446788710000145}}.

\bibitem{maclane:98}
Saunders Mac~{L}ane.
\newblock {\em Categories for the Working Mathematician}.
\newblock Graduate Texts in Mathematics, Vol. 5. Springer-Verlag, 2nd edition,
  1998.

\bibitem{phdthesis}
Brett McLean.
\newblock {\em Algebras of Partial Functions}.
\newblock PhD thesis, University College London, June 2018.

\bibitem{Pippenger_1997}
Nick Pippenger.
\newblock Regular languages and {S}tone duality.
\newblock {\em Theory of Computing Systems}, 30(2):121--134, 1997.
\newblock \href {https://doi.org/10.1007/bf02679444}
  {\path{doi:10.1007/bf02679444}}.

\bibitem{schutzenberger1965finite}
Marcel-Paul Sch\"utzenberger.
\newblock On finite monoids having only trivial subgroups.
\newblock {\em Information and Control}, 8(2):190--194, 1965.
\newblock \href {https://doi.org/10.1016/S0019-9958(65)90108-7}
  {\path{doi:10.1016/S0019-9958(65)90108-7}}.

\end{thebibliography}
\end{document}